\documentclass[final,leqno]{siamltex704}
\usepackage{mathrsfs}
\usepackage{amsmath}
\usepackage{amssymb}
\usepackage{graphicx}
\usepackage[notcite,notref]{showkeys}
\usepackage{hyperref}

\newtheorem{algorithm}{Weak Galerkin Algorithm}
\newcommand{\bq}{\begin{equation}}
\newcommand{\eq}{\end{equation}}

\def\T{{\mathcal T}}
\def\E{{\mathcal E}}
\def\pT{{\partial T}}
\def\l{{\langle}}
\def\r{{\rangle}}

\def\bn{{\bf n}}
\def\bq{{\bf q}}

\def\3bar{{|\hspace{-.02in}|\hspace{-.02in}|}}

\title{A numerical study on the weak Galerkin method for the Helmholtz equation}
\author{lin Mu\thanks{Department of Mathematics,
Michigan State University, East Lansing, MI 48824, (linmu@math.msu.edu).}
\and Junping Wang\thanks{Division of Mathematical Sciences, National Science
Foundation, Arlington, VA 22230 (jwang@\break nsf.gov). The research
of Wang was supported by the NSF IR/D program, while working at the
Foundation. However, any opinion, finding, and conclusions or
recommendations expressed in this material are those of the author
and do not necessarily reflect the views of the National Science
Foundation.} \and Xiu Ye\thanks{Department of Mathematics and
Statistics, University of Arkansas at Little Rock, Little Rock, AR
72204 (xxye@ualr.edu). This research of Ye was supported in part by
National Science Foundation Grant DMS-1115097.} \and Shan
Zhao\thanks{Department of Mathematics, University of Alabama,
Tuscaloosa, AL 35487 (szhao@as.ua.edu). The research of Zhao was
supported in part by National Science Foundation Grant DMS-1016579.
}}
\begin{document}
\maketitle

\begin{abstract}
A weak Galerkin (WG) method is introduced and numerically tested for
the Helmholtz equation. This method is flexible by using
discontinuous piecewise polynomials and retains the mass
conservation property. At the same time, the WG finite element
formulation is symmetric and parameter free. Several test scenarios
are designed for a numerical investigation on the accuracy,
convergence, and robustness of the WG method in both inhomogeneous
and homogeneous media over convex and non-convex domains.
Challenging problems with high wave numbers are also examined. Our
numerical experiments indicate that the weak Galerkin is a finite
element technique that is easy to implement, and provides very
accurate and robust numerical solutions for the Helmholtz problem
with high wave numbers.
\end{abstract}


\pagestyle{myheadings}

\section{Introduction}
We consider the Helmholtz equation of the form
\begin{eqnarray}
-\nabla \cdot (d \nabla u) -\kappa^2u &=& f, \quad
\mbox{in}\;\Omega,\label{pde}\\
d\nabla u \cdot \bn - i \kappa u &=& g,\quad \mbox{on}\; \partial\Omega,\label{bc}
\end{eqnarray}
where $\kappa>0$ is the wave number, $f\in L^2(\Omega)$ represents a harmonic source,
$g\in L^2(\partial\Omega)$ is a given data function, and
$d=d(x,y)>0$ is a spatial function describing the dielectric
properties of the medium. Here $\Omega$ is a polygonal or polyhedral domain in
$\mathbb{R}^d\; (d=2,3)$.

Under the assumption that the time-harmonic behavior is assumed,
the Helmholtz equation (\ref{pde})
governs many macroscopic wave phenomena in the frequency domain
including wave propagation, guiding, radiation and scattering.
The numerical solution to the
Helmholtz equation plays a vital role in a wide range of
applications in electromagnetics, optics, and acoustics, such as
antenna analysis and synthesis, radar cross section calculation,
simulation of ground or surface penetrating radar, design of
optoelectronic devices, acoustic noise control, and seismic wave
propagation. However, it remains a challenge to design robust and
efficient numerical algorithms for the Helmholtz equation,
especially when large wave numbers or highly oscillatory solutions
are involved \cite{Zienkiewicz}.

For the Helmholtz  problem (\ref{pde})-(\ref{bc}), the corresponding
variational form is given by seeking $u\in H^1(\Omega)$ satisfying
\begin{equation}\label{wf}
(d\nabla u,\nabla
v)-\kappa^2(u,\;v)+i\kappa\l u,\;v\r_{\partial\Omega}= (f, v)
+\l g,v\r_{\partial\Omega},\qquad \forall v\in H^1(\Omega),
\end{equation}
where $(v,w)=\int_\Omega vwdx$ and
$\l v,w\r_{\partial \Omega}=\int_{\partial \Omega}vwds$.
In a classic finite element procedure, continuous polynomials are used to approximate the true solution $u$. In many situations, the use of discontinuous functions in the finite element approximation often provides the methods with
much needed flexibility to handle more complicated practical problems. However, for discontinuous polynomials, the strong gradient $\nabla$ in (\ref{wf}) is no longer meaningful. Recently developed weak Galerkin finite element methods \cite{wy}
provide means to solve this difficulty by replacing the differential operators by the weak forms as distributions for discontinuous approximating functions.

Weak Galerkin (WG) methods refer to general finite element techniques for partial differential equations and were  first introduced and analyzed in \cite{wy} for second order elliptic equations.
Through rigorous error analysis,
optimal order of convergence of the WG solution in both
discrete $H^1$ norm and $L^2$ norm is established under
minimum regularity assumptions in \cite{wy}. The mixed weak Galerkin finite element method is studied in \cite{wy-mixed}.
The WG methods are by design using discontinuous approximating functions.

In this paper, we will apply WG finite element methods \cite{wy}
to the Helmholtz equation.
The WG finite element approximation to (\ref{wf}) can be derived naturally by simply replacing the differential operator gradient $\nabla$ in (\ref{wf}) by a weak gradient $\nabla_w$: find
$u_h\in V_h$ such that for all $v_h\in V_h$ we have
\begin{equation}\label{wg1}
(d\nabla_w u_h,\nabla_w
v_h)-\kappa^2(u_0,\;v_0)+i\kappa\l u_b,\;v_b\r_{\partial\Omega}= (f, v_0)
+\l g,v_b\r_{\partial\Omega},
\end{equation}
where $u_0$ and $u_b$ represent the values of $u_h$ in the interior and the boundary of the triangle respectively.  The weak gradient $\nabla_w$ will be defined precisely in the next section. We note that the weak Galerkin finite element formulation (\ref{wg1}) is simple, symmetric and parameter free.

To fully explore the potential of the WG finite element formulation (\ref{wg1}),
we will investigate its performance for solving the Helmholtz
problems with large wave numbers.
It is well known that the numerical
performance of any finite element solution to the Helmholtz equation
depends significantly on the wave number $k$. When $k$ is very large
-- representing a highly oscillatory wave, the mesh size $h$ has to be
sufficiently small for the scheme to resolve the oscillations. To
keep a fixed grid resolution, a natural rule is to choose $kh$ to be
a constant in the mesh refinement, as the wave number $k$ increases
\cite{IhlBabH,bao04}. However, it is known
\cite{IhlBabH,IhlBabHP,BIPS,BabSau} that, even under such a mesh
refinement, the errors of continuous Galerkin finite element
solutions deteriorate rapidly when $k$ becomes larger. This
non-robust behavior with respect to $k$ is known as the ``pollution
effect''.

To the end of alleviating the pollution effect,
various continuous or discontinuous finite element methods have been
developed in the literature for solving the Helmholtz equation
with large wave numbers
\cite{BIPS,BabSau,Melenk,Monk99,UWVF98,UWVF03,Farhat03,Farhat04,fw,abcm,Chung,cdg,cgl,Monk11}.
A commonly used strategy in these effective finite element methods is
to include some analytical knowledge of the Helmholtz equation, such
as characteristics of traveling plane wave solutions,
asymptotic solutions or fundamental solutions, into the finite
element space.
Likewise, analytical
information has been incorporated in the basis functions  of the boundary
element methods to address the high frequency problems
\cite{Giladi,Langdon06,Langdon07}. On the other hand,
many spectral methods, such as local spectral methods \cite{bao04},
spectral Galerkin methods \cite{shen05,shen07}, and spectral element
methods \cite{Heikkola,Ainsworth09} have also been developed for solving
the Helmholtz equation with large wave numbers.
Pollution effect can be effectively controlled in these spectral
type collocation or Galerkin formulations, because
the pollution error is directly related to the dispersion error,
i.e., the phase difference between the numerical and exact waves
\cite{IhlBab,Ainsworth04}, while the spectral methods typically
produce negligible dispersive errors.

The objective of the present paper is twofold.
First, we will introduce weak Galerkin methods for the Helmholtz equation.
The second aim of the  paper is to
investigate the performance of the WG methods for solving the Helmholtz equation
with high wave numbers.
To demonstrate the potential of the WG finite element methods in solving high frequency problems,
we will not attempt to build the analytical knowledge into the WG formulation (\ref{wg1})
and we will restrict ourselves to low order WG elements.
We will investigate the robustness and effectiveness of such plain WG methods
through many carefully designed numerical experiments.

The rest of this paper is organized as follows. In Section 2, we
will introduce a weak Galerkin finite element formulation for the
Helmholtz equation by following the idea presented in \cite{wy}. Implementation of the WG method for the problem (\ref{pde})-(\ref{bc}) is discussed in Section 3.  In Section 4, we shall
present some numerical results obtained from the weak Galerkin
method with various orders.
Finally, this paper ends with some concluding remarks.

\section{A Weak Galerkin Finite Element Method}

Let ${\cal T}_h$ be a partition of the domain $\Omega$ with mesh size
$h$. Assume that the partition ${\cal T}_h$ is shape regular so that
the routine inverse inequality in the finite element analysis holds
true (see \cite{ci}).  Denote by
$P_k(T)$ the set of polynomials in $T$ with degree no more than
$k$, and $P_k(e)$, $e\in {\partial T}$, the set of polynomials on each segment
(edge or face) of $\partial T$ with degree no more than $k$.

For $k\ge 0$ and given $\T_h$, we define the weak Galerkin (WG)  finite element space as follows
\begin{equation}\label{vh}
V_h=\left\{ v=\{v_0, v_b\}\in L^2(\Omega):\ \{v_0, v_b\}|_{T}\in P_k(T)\times P_k(e),e\in\partial T, \forall T\in {\cal T}_h \right\},
\end{equation}
where $v_0$ and $v_b$ are the values of $v$ restricted on the interior of element $T$ and the boundary of element $T$ respectively. Since $v_b$ may not necessarily be related to the trace of $v_0$ on $\partial T$, we write $v=\{v_0,v_b\}$.  For a given $T\in\T_h$, we define another vector space
\[
RT_k(T)=P_k(T)^d+\tilde{P}_k(T){\bf x},
\]
where $\tilde{P}_k(T)$ is the set of homogeneous polynomials of degree $k$ and ${\bf x}=(x_1,\cdots, x_d)$ (see \cite{bf}). We will find a locally defined discrete weak gradient from this space on each element $T$.

The main idea of the weak Galerkin method is to introduce weak derivatives for discontinuous functions and to use them in discretizing the corresponding variational forms such as (\ref{wf}).
The differential operator used in (\ref{wf}) is a gradient.  A weak gradient has been defined in \cite{wy}.  Now we define approximations  of the weak gradient as follows. For each $v=\{v_0, v_b\} \in V_h$, we define a discrete weak gradient $\nabla_{w} v\in RT_k(T)$ on each
element $T$ such that
\begin{equation}\label{d-g}
(\nabla_{w}v,\  \tau)_T = -(v_0,\  \nabla\cdot \tau)_T+
\l v_b, \ \tau\cdot\bn \r_\pT,\quad\quad\forall \tau\in RT_k(T),
\end{equation}
where $\nabla_{w}v$ is locally defined on each element $T$, $(v,w)_T=\int_T vwdx$ and
$\l v,w\r_{\partial T}=\int_{\partial T}vwds$. We will  use  $(\nabla_w v,\ \nabla_w w)$ to denote $\sum_{T\in\T_h}(\nabla_w v,\ \nabla_w w)_T$. Then the WG method for the Helmholtz equation (\ref{pde})-(\ref{bc}) can be stated as follows.

\smallskip

\begin{algorithm}
A numerical approximation for (\ref{pde}) and (\ref{bc}) can be
obtained by seeking $u_h=\{u_0,u_b\}\in V_h$  such that for all $v_h=\{v_0,v_b\}\in V_h$
\begin{equation}\label{WG}
(d\nabla_w u_h,\nabla_w
v_h)-\kappa^2(u_0,\;v_0)+i\kappa\l u_b,\;v_b\r_{\partial\Omega}= (f, v_0)
+\l g,v_b\r_{\partial\Omega}.
\end{equation}
\end{algorithm}

Denote by $Q_h u=\{Q_0 u,\;Q_bu\}$ the $L^2$ projection onto
$P_k(T)\times P_{k}(e)$, $e\in\partial T$. In other words, on each element
$T$, the function $Q_0 u$ is defined as the $L^2$ projection of $u$
in $P_k(T)$ and $Q_b u$ is the $L^2$ projection of $u$ in
$P_{k}(\partial T)$.

For equation (\ref{pde}) with Dirichlet boundary condition $u=g$ on $\partial\Omega$, optimal error estimates have been obtained in \cite{wy}.

For a sufficiently small mesh size $h$, we can derive following optimal error estimate for the Helmholtz equation (\ref{pde}) with the mixed boundary condition (\ref{bc}).

\begin{theorem} Let $u_h\in V_h$ and $u\in H^{k+2}(\Omega)$  be the solutions of (\ref{WG}) and (\ref{pde})-(\ref{bc}) respectively and assume that $\Omega$ is convex.
Then for $k\ge 0$,
there exists a constant $C$ such that
\begin{eqnarray}
\|\nabla_w( u_h-Q_hu)\| &\le&  Ch^{k+1}(\|u\|_{k+2}+\|f\|_k),\label{err1}\\
\|u_h-Q_hu\| &\le & Ch^{k+2}(\|u\|_{k+2}+\|f\|_k).\label{err2}
\end{eqnarray}
\end{theorem}

\begin{proof}
The proof of this theorem is similar to that of Theorem  8.3 and Theorem 8.4 in
\cite{wy} and is very long. Since the emphasis of this paper is to investigate the performance of the WG method, we will omit details of the proof.
\end{proof}

\section{Implementation of WG method}
First, define a bilinear form $a(\cdot,\cdot)$ as
\[
a(u_h,v_h)=(d\nabla_w u_h,\nabla_w v_h)-\kappa^2(u_0,\;v_0)+i\kappa\l u_b,\;v_b\r_{\partial\Omega}.
\]
Then (\ref{WG}) can be rewritten with $v_h=\{v_0,v_b\}$
\begin{equation}\label{wg8}
a(u_h,v_h)=(f, v_0)+\l g,v_b\r_{\partial\Omega}.
\end{equation}

The methodology of implementing the WG  methods is the same as that for continuous
Galerkin finite element methods except that the standard gradient
operator $\nabla$ should be replaced by the discrete weak gradient operator
$\nabla_w$.

In the following, we will use the lowest order weak Galerkin element
($k$=0) on triangles as an example to demonstrate how one might implement the weak
Galerkin finite element method for solving the Helmholtz problem
(\ref{pde}) and (\ref{bc}). Let $N(T)$ and $N(e)$ denote, respectively, the number of
triangles and the number of edges associated with a triangulation
${\cal T}_h$.  Let ${\cal E}_h$ denote the union of the boundaries
of the triangles $T$ of ${\cal T}_h$.

The procedure of implementing the WG method (\ref{WG}) consists of
the following three steps.

\begin{enumerate}
\item Find basis functions for $V_h$ defined in (\ref{vh}):
\begin{eqnarray*}
V_h&=&{\rm span} \{\phi_1,\cdots,\phi_{N(T)},\psi_1,\cdots,\psi_{N(e)}\}
={\rm span} \{\Phi_1,\cdots,\Phi_{n}\}
\end{eqnarray*}
where $n=N(T)+N(e)$ and
\[
\phi_i=\left\{
\begin{array}{l}
  1
  \quad
  \mbox{on} \;\; T_i,
  \\ [0.08in]
  0
  \quad
  \mbox{otherwise},
\end{array}
\right.
\psi_j=\left\{
\begin{array}{l}
  1
  \quad
  \mbox{on} \;\; e_j,
  \\ [0.08in]
  0
  \quad
  \mbox{otherwise},
\end{array}
\right.
\]
for $T_i\in\T_h$ and $e_j\in \E_h$. Please note that $\phi_i$ and $\psi_j$ are defined on whole $\Omega$.\\

\item Substituting $u_h=\sum_{j=1}^{n}\alpha_j\Phi_j$ into (\ref{wg8}) and letting $v=\Phi_i$ in (\ref{wg8}) yield
\begin{equation}\label{sys}
\sum_{j=1}^na(\Phi_j,\Phi_i)\alpha_j=(f, \Phi_i^0)+\l g,\Phi_i^b\r_{\partial\Omega},\quad i=1,\cdots n
 \end{equation}
where $\Phi_i^0$ and $\Phi_i^b$ are the values of $\Phi_i$ on the interior of the triangle and the boundary of the triangle respectively. In our computations, the integrations on the right-hand side
of (\ref{sys}) are conducted numerically. In particular,
a 7-points two-dimensional Gaussian quadrature
and a 3-points one-dimensional Gaussian quadrature are employed, respectively, to calculate
$(f, \Phi_i^0)$ and $\l g,\Phi_i^b\r_{\partial\Omega}$ numerically.\\

\item Form the coefficient matrix $(a(\Phi_j,\Phi_i))_{i,j}$ of the linear system (\ref{sys}) by computing
\begin{equation}\label{bilinear}
a(\Phi_j,\Phi_i)=(d\nabla_w \Phi_j,\nabla_w \Phi_i)-\kappa^2(\Phi_j^0,\;\Phi_i^0)+i\kappa\l \Phi_j^b,\;\Phi_i^b\r_{\partial\Omega}.
\end{equation}
All integrations in (\ref{bilinear}) are carried out analytically.

\end{enumerate}

Finally, we will explain how to compute the weak gradient $\nabla_w$ for a given function $v\in V_h$ when $k=0$.
For a given $T\in\T_h$, we will find  $\nabla_w v\in RT_0(T)$,
\[
RT_0(T)=\left(\begin{array}{c} a+cx \\b+cy\\\end{array}\right)={\rm span}
\{\theta_1,\theta_1,\theta_3\}.
\]
For example, we can choose $\theta_i$ as follows
\[
\theta_1=\left(\begin{array}{c}1 \\0 \\\end{array}\right), \theta_2=\left(\begin{array}{c}0 \\1 \\\end{array}\right),
\theta_3=\left(\begin{array}{c}x \\y \\\end{array}\right).
\]
Thus on each element $T\in {\cal T}_h$, $\nabla_w v=\sum_{j=1}^3c_j\theta_j$. Using the definition of the discrete weak gradient (\ref{d-g}), we find  $c_j$ by solving the following linear system:
\[
\left(
  \begin{array}{ccc}
    (\theta_1, \theta_1)_T & (\theta_2, \theta_1)_T & (\theta_3, \theta_1)_T \\
    (\theta_1, \theta_2)_T& (\theta_2, \theta_2)_T& (\theta_3, \theta_2)_T \\
    (\theta_1, \theta_3)_T& (\theta_2, \theta_3)_T& (\theta_3, \theta_3)_T\\
  \end{array}
\right)
\left(\begin{array}{c}c_1 \\c_2 \\c_3 \\\end{array}\right)=\left(\begin{array}{c}
-(v_0,\nabla\cdot\theta_1)_T+\l v_b,\;\theta_1\cdot\bn\r_{\partial T} \\
 -(v_0,\nabla\cdot\theta_2)_T+\l v_b,\;\theta_2\cdot\bn\r_{\partial T}\\
 -(v_0,\nabla\cdot\theta_3)_T+\l v_b,\;\theta_3\cdot\bn\r_{\partial T}
 \end{array}
 \right).
 \]
The inverse of the above coefficient matrix can be obtained
explicitly or numerically through a local matrix solver. For the basis function $\Phi_i$, $\nabla_w\Phi_i$ is nonzero on only  one or two triangles.

\section{Numerical Experiments}
In this section, we examine the WG method by testing its accuracy,
convergence, and robustness for solving two dimensional Helmholtz
equations. The pollution effect due to large wave numbers will be
particularly investigated and tested numerically. For convergence
tests, both piecewise constant and piecewise linear finite elements
will be considered. To demonstrate the robustness of the WG method, the Helmholtz
equation in both homogeneous and inhomogeneous media will be solved
on convex and non-convex computational domains.
The mesh generation and all computations are conducted in the MATLAB
environment.
For simplicity, a
structured triangular mesh is employed in all cases, even though the WG method
is known to be very flexible in dealing with various different finite
element partitions \cite{mwy-PolyRedu,mwy-biharmonic}.

Two types of relative errors are measured in our numerical
experiments. The first one is the relative $L^2$ error
defined by
$$ \frac{\| u_h - Q_h u \|}{\| Q_h u \|}.$$
The second one is the relative $H^1$ error
defined in terms of the discrete gradient
$$ \frac{\| \nabla_w (u_h - Q_h u) \|}{\| \nabla_w Q_h u \|}.$$
Numerically, the $H^1$-semi-norm will be calculated
as
$$
\3bar u_h-Q_hu \3bar^2=h^{-1}\langle
u_0-u_b-(Q_0u-Q_bu),u_0-u_b-(Q_0u-Q_bu)\rangle_{\partial \Omega}
$$
for the lowest order finite element (i.e., piecewise constants). For
piecewise linear elements, we use the original definition of
$\nabla_w$ to compute the $H^1$-semi-norm $\|\nabla_w (u_h - Q_h
u)\|$.

\begin{figure}[!tb]
\centering
\begin{tabular}{cc}
  \resizebox{2.25in}{2.25in}{\includegraphics{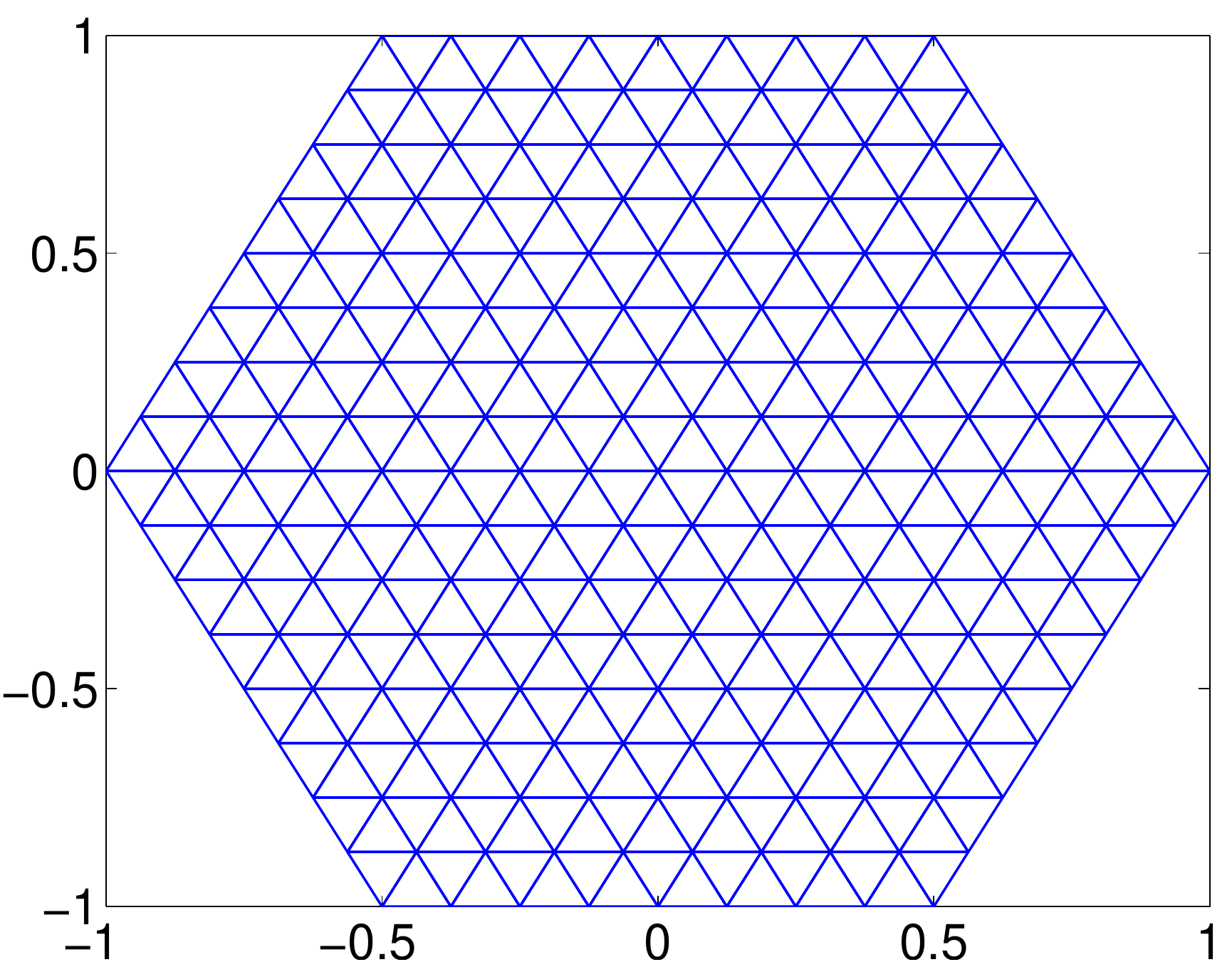}} \quad
  \resizebox{2.25in}{2.25in}{\includegraphics{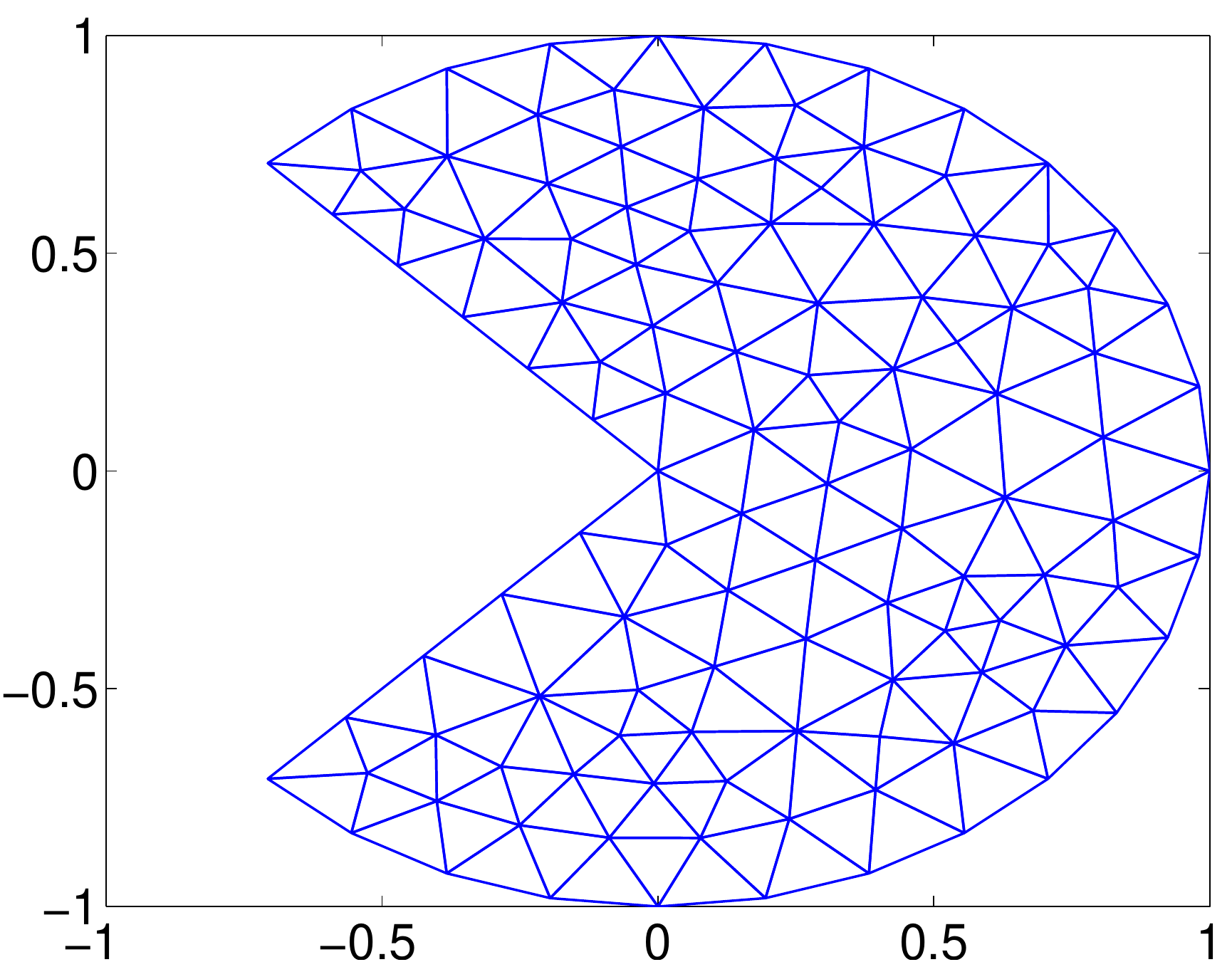}}
\end{tabular}
\caption{Geometry of testing domains and sample meshes. Left: a
convex hexagon domain; Right: a non-convex imperfect circular
domain.} \label{fig.domain}
\end{figure}

\subsection{A convex Helmholtz problem}\label{convex}
We first consider a homogeneous Helmholtz equation defined on a
convex hexagon domain, which has been studied in \cite{fw}. The
domain $\Omega$ is the unit regular hexagon domain centered at the
origin $(0,0)$, see Fig. \ref{fig.domain} (left). Here we set $d=1$
and $f=\sin(kr)/r$ in (\ref{pde}), where $r=\sqrt{x^2+y^2}$. The
boundary data $g$ in the Robin boundary condition (\ref{bc}) is
chosen so that the exact solution is given by
\begin{eqnarray}
u=\frac{\cos(kr)}{k}-\frac{\cos k+i\sin k}{k(J_0(k)+iJ_1(k))}J_0(kr)
\end{eqnarray}
where $J_{\xi}(z)$ are Bessel functions of the first kind. Let
$\T_h$ denote the regular triangulation that consists of $6N^2$
triangles of size $h=1/N$, as shown in Fig. \ref{fig.domain} (left)
for $T_{\frac18}$.

\begin{table}[!t]
\caption{Convergence of piecewise constant WG for the Helmholtz equation on a
convex domain with wave number $k=1$.} \label{table.Ex1k1}
\begin{center}
\begin{tabular}{|l|l|l|l|l|}
\hline
& \multicolumn{2}{c|}{relative $H^1$} & \multicolumn{2}{c|}{relative $L^2$} \\
\cline{2-3} \cline{4-5}
$h$ & error & order & error & order \\
\hline
  5.00e-01 &  2.49e-02 &      &  4.17e-03 &\\
  2.50e-01 &  1.11e-02 & 1.16 &  1.05e-03 & 1.99 \\
  1.25e-01 &  5.38e-03 & 1.05 &  2.63e-04 & 2.00 \\
  6.25e-02 &  2.67e-03 & 1.01 &  6.58e-05 & 2.00 \\
  3.13e-02 &  1.33e-03 & 1.00 &  1.64e-05 & 2.00 \\
  1.56e-02 &  6.65e-04 & 1.00 &  4.11e-06 & 2.00 \\
\hline
\end{tabular}
\end{center}
\end{table}

Table \ref{table.Ex1k1} illustrates the performance of the WG method
with piecewise constant elements for the Helmholtz equation with
wave number $k=1$. Uniform triangular partitions were used in the
computation through successive mesh refinements. The relative errors
in $L^2$ norm and $H^1$ semi-norm can be seen in Table
\ref{table.Ex1k1}. The Table also includes numerical estimates for
the rate of convergence in each metric. It can be seen that the
order of convergence in the relative $H^1$ semi-norm and relative
$L^2$ norm are, respectively, one and two for piecewise constant
elements.

\begin{table}[!t]
\caption{ Convergence of piecewise linear WG for the Helmholtz equation on a
convex domain with wave number $k=5$.} \label{table.Ex1k5}
\begin{center}
\begin{tabular}{|l|l|l|l|l|}
\hline
& \multicolumn{2}{c|}{relative $H^1$} & \multicolumn{2}{c|}{relative $L^2$} \\
\cline{2-3} \cline{4-5}
$h$ & error & order & error & order \\
\hline
  2.50e-01 &  9.48e-03 &      & 2.58e-04 & \\
  1.25e-01 &  2.31e-03 & 2.04 & 3.46e-05 & 2.90 \\
  6.25e-02 &  5.74e-04 & 2.01 & 4.47e-06 & 2.95 \\
  3.13e-02 &  1.43e-04 & 2.00 & 5.64e-07 & 2.99 \\
  1.56e-02 &  3.58e-05 & 2.00 & 7.06e-08 & 3.00 \\
  7.81e-03 &  8.96e-06 & 2.00 & 8.79e-09 & 3.01 \\
\hline
\end{tabular}
\end{center}
\end{table}

High order of convergence can be achieved by using corresponding
high order finite elements in the present WG framework. To
demonstrate this phenomena, we consider the same Helmholtz problem
with a slightly larger wave number $k=5$. The WG with piecewise
linear functions was employed in the numerical approximation. The
computational results are reported in Table \ref{table.Ex1k5}. It is
clear that the numerical experiment validates the theoretical
estimates. More precisely, the rates of convergence in the relative
$H^1$ semi-norm and relative $L^2$ norm are given by two and three,
respectively.

\subsection{A non-convex Helmholtz problem}

We next explore the use of the WG method for solving a Helmholtz
problem defined on a non-convex domain, see Fig. \ref{fig.domain}
(right). The medium is still assumed to be homogeneous, i.e., $d=1$
in (\ref{pde}). We are particularly interested in the performance of
the WG method for dealing with the possible field singularity at the
origin. For simplicity, only the piecewise constant $RT_0$ elements
are tested for the present problem. Following \cite{Monk11}, we take
$f=0$ in (\ref{pde}) and the boundary condition is simply taken as a
Dirichlet one: $u=g$ on $\partial \Omega$. Here $g$ is prescribed
according to the exact solution \cite{Monk11}
\begin{equation}\label{solution2}
u= J_{\xi}( k \sqrt{x^2+y^2}) \cos (\xi \arctan( y/x)).
\end{equation}

\begin{figure}[!tb]
\centering
\begin{tabular}{cc}
  \resizebox{2.55in}{2.1in}{\includegraphics{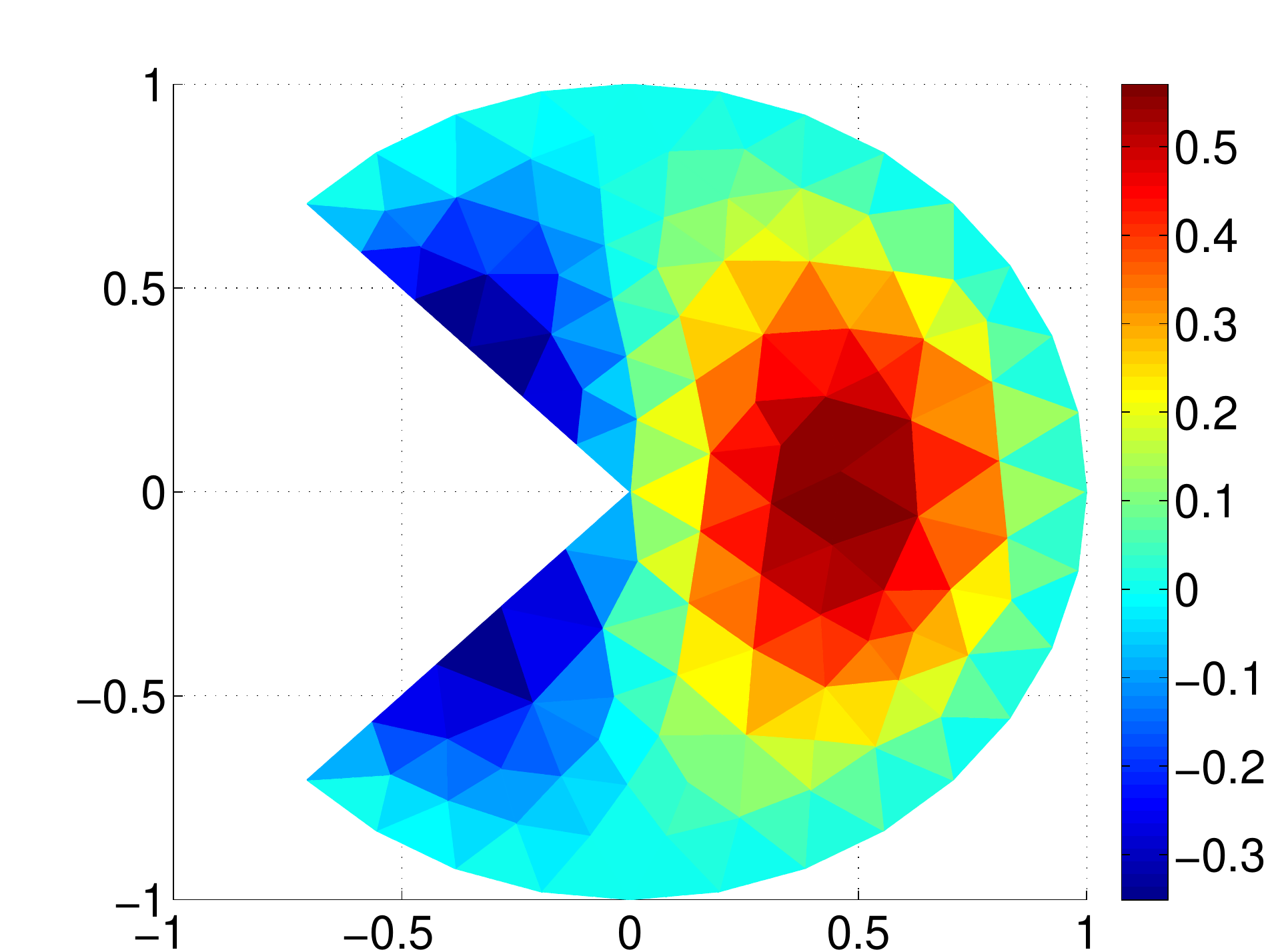}} \quad
  \resizebox{2.45in}{2.1in}{\includegraphics{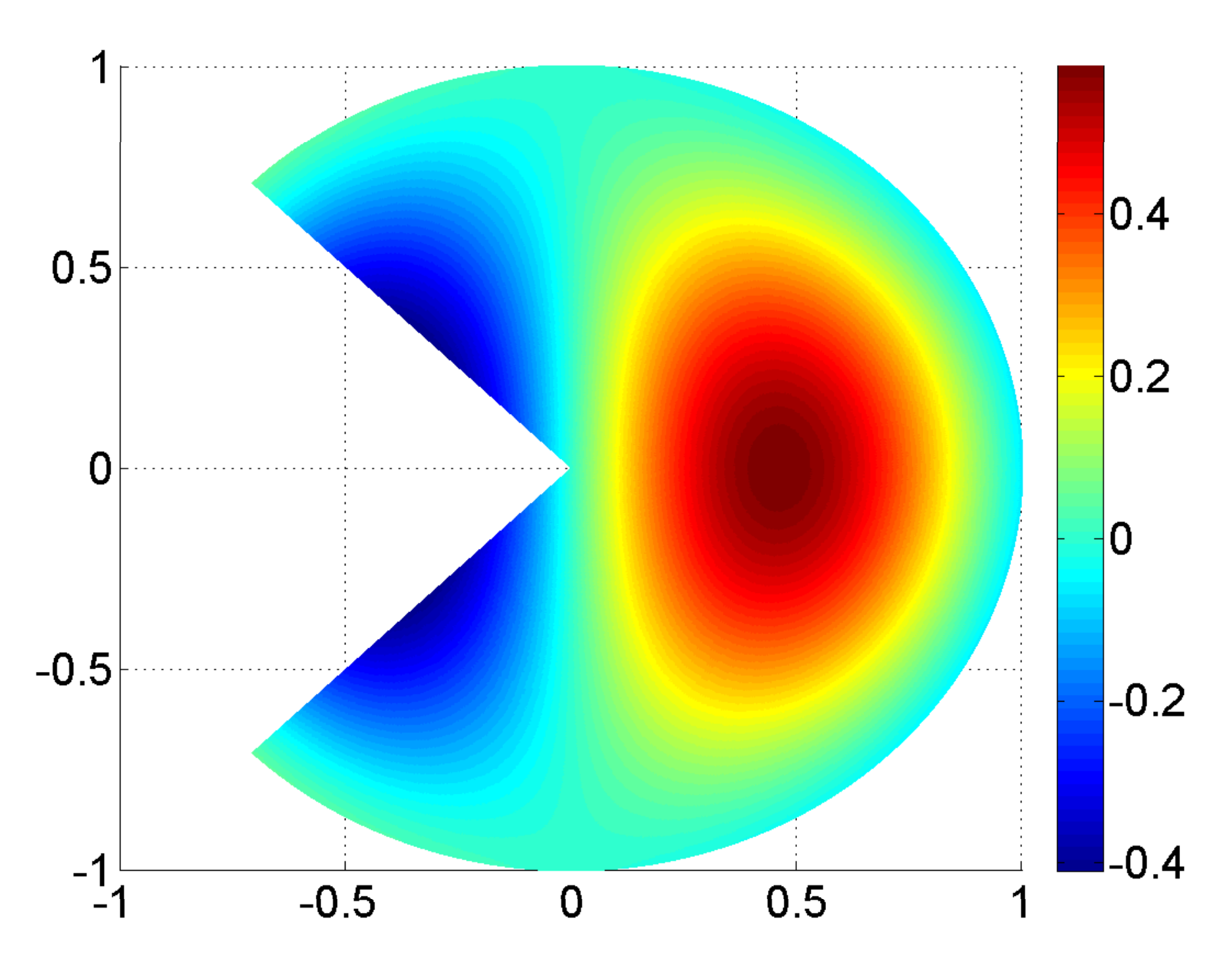}}
\end{tabular}
\caption{WG solutions for the non-convex Helmholtz problem with
$k=4$ and $\xi=1$. Left: Mesh level $1$; Right: Mesh level $6$.}
\label{fig.Ex2xi1}
\end{figure}

\begin{figure}[!tb]
\centering
\begin{tabular}{cc}
  \resizebox{2.55in}{2.1in}{\includegraphics{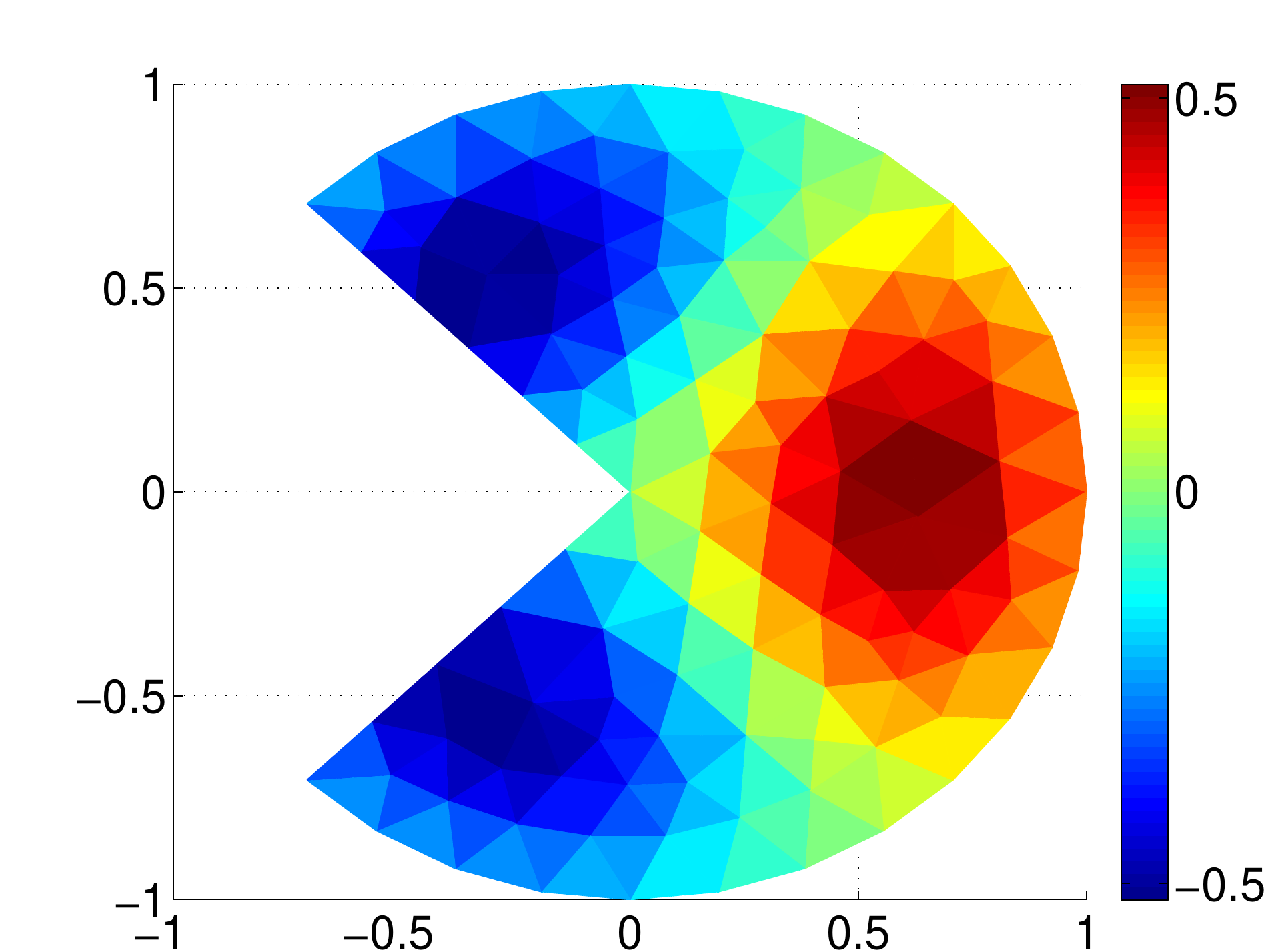}} \quad
  \resizebox{2.45in}{2.1in}{\includegraphics{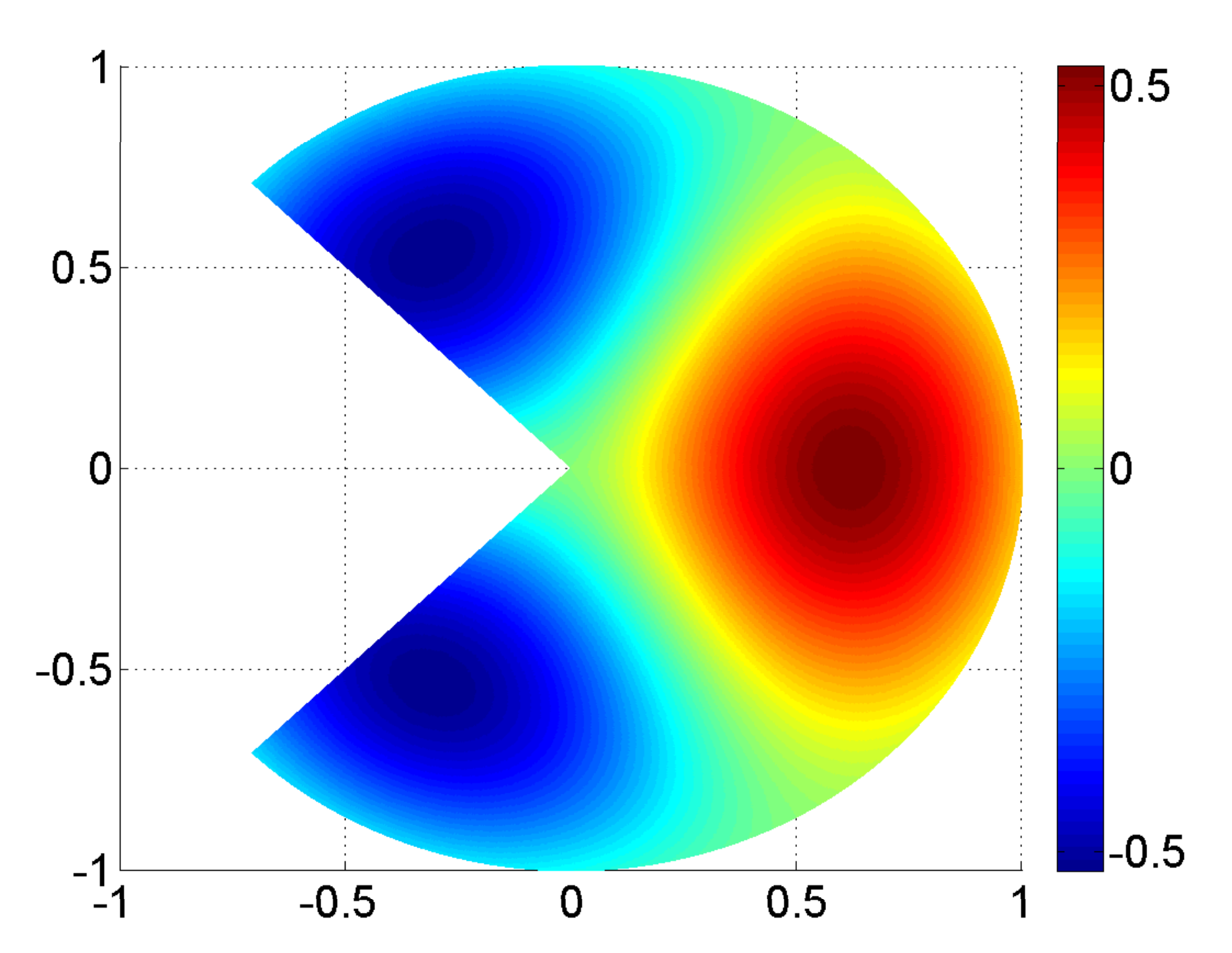}}
\end{tabular}
\caption{WG solutions for the non-convex Helmholtz problem with
$k=4$ and $\xi=3/2$. Left: Mesh level $1$; Right: Mesh level $6$.}
\label{fig.Ex2xi32}
\end{figure}

\begin{figure}[!tb]
\centering
\begin{tabular}{cc}
  \resizebox{2.55in}{2.1in}{\includegraphics{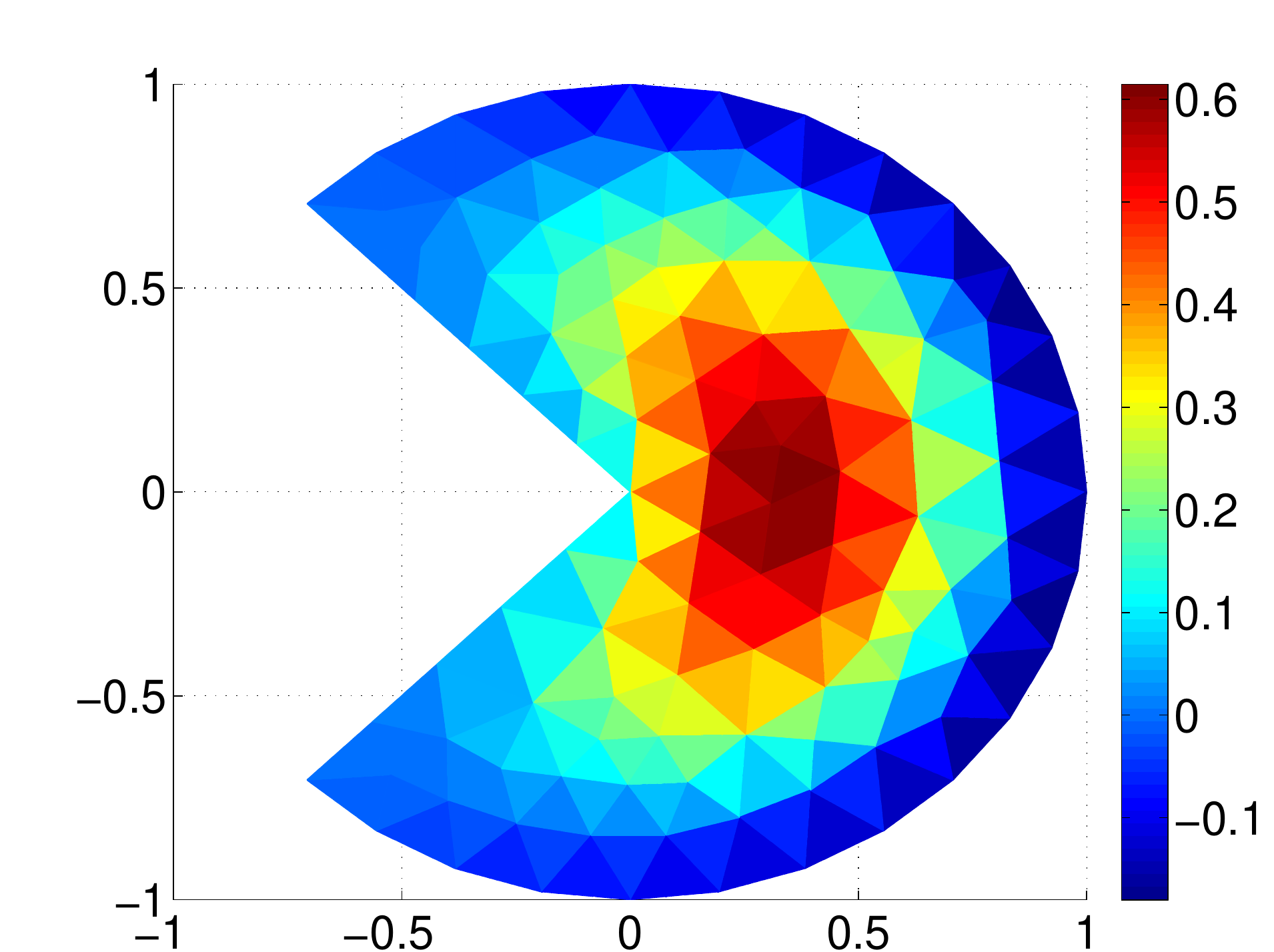}}
  \resizebox{2.45in}{2.1in}{\includegraphics{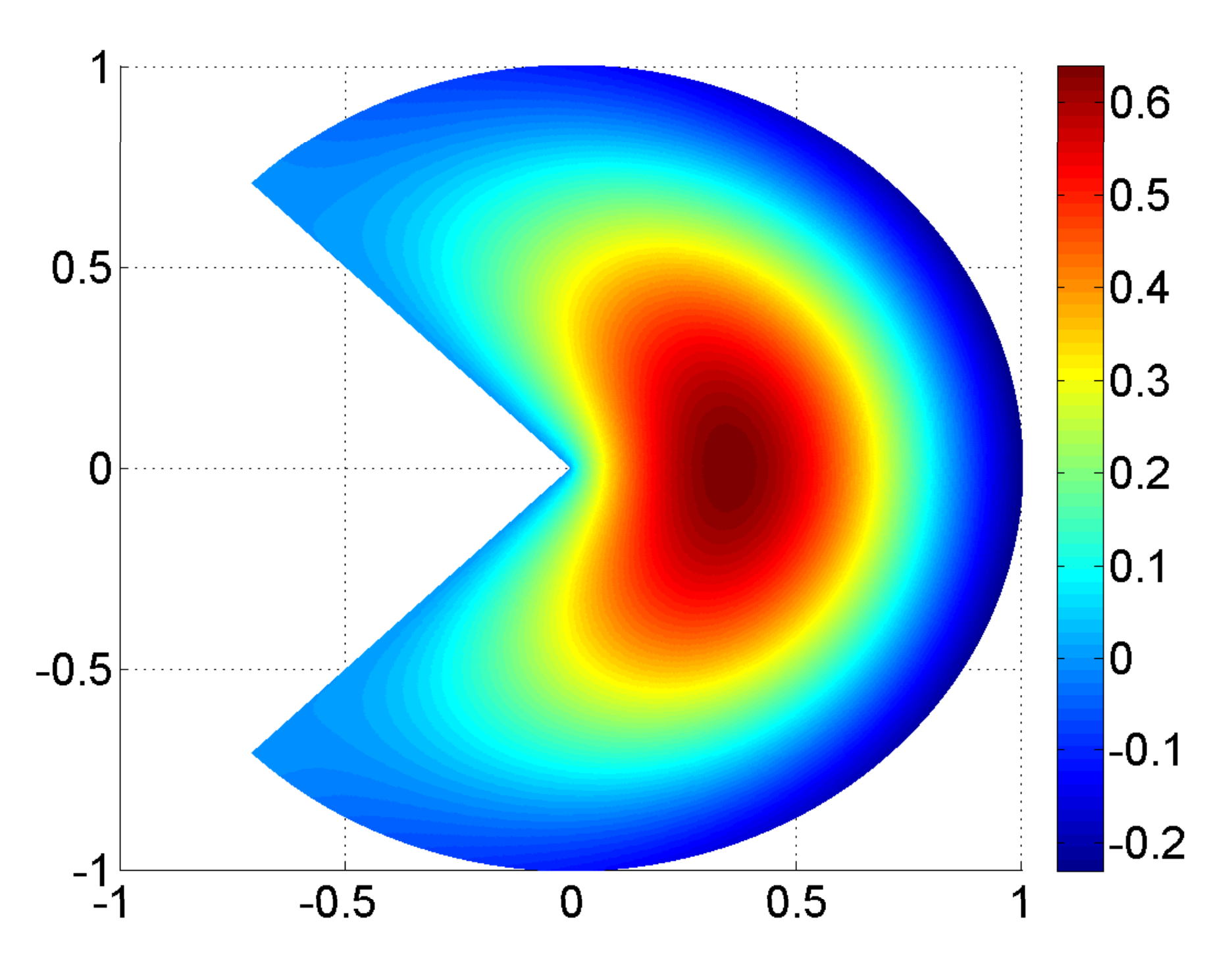}}
\end{tabular}
\caption{WG solutions for the non-convex Helmholtz problem with
$k=4$ and $\xi=2/3$. Left: Mesh level $1$; Right: Mesh level $6$.}
\label{fig.Ex2xi23}
\end{figure}

In the present study, the wave number was chosen as $k=4$ and three
values for the parameter $\xi$ are considered; i.e., $\xi=1$,
$\xi=3/2$ and $\xi=2/3$. The same triangular mesh is used in the WG
method for all three cases. In particular, an initial mesh is first
generated by using MATLAB with default settings, see Fig.
\ref{fig.domain} (right). Next, the mesh is refined uniformly for
five times. The WG solutions on mesh level $1$ and mesh level $6$
are shown in Fig. \ref{fig.Ex2xi1}, Fig. \ref{fig.Ex2xi32}, and Fig.
\ref{fig.Ex2xi23}, respectively, for $\xi=1$, $\xi=3/2$, and
$\xi=2/3$. Since the numerical errors are quite small for the WG
approximation corresponding to mesh level $6$, the field modes
generated by the densest mesh are visually indistinguishable from
the analytical ones. In other words, the results shown in the right
charts of Fig. \ref{fig.Ex2xi1}, Fig. \ref{fig.Ex2xi32}, and Fig.
\ref{fig.Ex2xi23} can be regarded as analytical results. It can be
seen that in all three cases, the WG solutions already agree with
the analytical ones at the coarsest level. Moreover, based on the
coarsest mesh, the constant function values can be clearly seen in
each triangle, due to the use of piecewise constant $RT_0$ elements.
Nevertheless, after the initial mesh is refined for five times, the
numerical plots shown in the right charts are very smooth. A perfect
symmetry with respect to the $x$-axis is clearly seen.

\begin{table}[!tb]
\caption{Numerical convergence test for the non-convex Helmholtz
problem with $k=4$ and $\xi=1$. } \label{table.Ex2xi1}
\begin{center}
\begin{tabular}{|l|l|l|l|l|}
\hline
& \multicolumn{2}{c|}{relative $H^1$} & \multicolumn{2}{c|}{relative $L^2$} \\
\cline{2-3} \cline{4-5}
$h$ & error & order & error & order \\
\hline
  2.44e-01 & 5.64e-02 &      & 1.37e-02 & \\
  1.22e-01 & 2.83e-02 & 1.00 & 3.56e-03 & 1.95 \\
  6.10e-02 & 1.42e-02 & 0.99 & 8.98e-04 & 1.99 \\
  3.05e-02 & 7.14e-03 & 1.00 & 2.25e-04 & 2.00 \\
  1.53e-02 & 3.57e-03 & 1.00 & 5.63e-05 & 2.00 \\
  7.63e-03 & 1.79e-03 & 1.00 & 1.41e-05 & 2.00 \\
\hline
\end{tabular}
\end{center}
\end{table}

\begin{table}[!tb]
\caption{Numerical convergence test for the non-convex Helmholtz
problem with $k=4$ and $\xi=3/2$. } \label{table.Ex2xi32}
\begin{center}
\begin{tabular}{|l|l|l|l|l|}
\hline
& \multicolumn{2}{c|}{relative $H^1$} & \multicolumn{2}{c|}{relative $L^2$} \\
\cline{2-3} \cline{4-5}
$h$ & error & order & error & order \\
\hline
  2.44e-01   & 5.56e-02 & & 1.12e-2& \\
  1.22e-01   & 2.81e-02 & 0.98 & 3.02e-03 & 1.89 \\
  6.10e-02   & 1.42e-02 & 0.99 & 8.06e-04 & 1.91 \\
  3.05e-02   & 7.14e-03 & 0.99 & 2.12e-04 & 1.92 \\
  1.53e-02   & 3.58e-03 & 1.00 & 5.54e-05 & 1.94 \\
  7.63e-03   & 1.79e-03 & 1.00 & 1.44e-05 & 1.95 \\
\hline
\end{tabular}
\end{center}
\end{table}

\begin{table}[!tb]
\caption{Numerical convergence test for the non-convex Helmholtz
problem with $k=4$ and $\xi=2/3$. } \label{table.Ex2xi23}
\begin{center}
\begin{tabular}{|l|l|l|l|l|}
\hline
& \multicolumn{2}{c|}{relative $H^1$} & \multicolumn{2}{c|}{relative $L^2$} \\
\cline{2-3} \cline{4-5}
$h$ & error & order & error & order \\
\hline
   2.44e-01   & 1.07e-01 &       &5.24e-02 & \\
   1.22e-01   & 5.74e-02 & 0.90  &2.18e-02 & 1.27 \\
   6.10e-02   & 3.23e-02 & 0.83  &9.01e-03 & 1.27 \\
   3.05e-02   & 1.89e-02 & 0.77  &3.68e-03 & 1.29 \\
   1.53e-02   & 1.14e-02 & 0.73  &1.49e-03 & 1.31 \\
   7.63e-03   & 6.99e-03 & 0.71  &5.96e-04 & 1.32 \\
\hline
\end{tabular}
\end{center}
\end{table}

We next investigate the numerical convergence rates for WG. The
numerical errors of the WG solutions for $\xi=1$, $\xi=3/2$ and
$\xi=2/3$ are listed, respectively, in Table \ref{table.Ex2xi1},
Table \ref{table.Ex2xi32}, and Table \ref{table.Ex2xi23}. It can be
seen that for $\xi=1$ and $\xi=3/2$, the numerical convergence rates
in the relative $H^1$ and $L^2$ errors remain to be first and second
order, while the convergence orders degrade for the non-smooth case
$\xi=2/3$. Mathematically, for both  $\xi=3/2$ and $\xi=2/3$, the
exact solutions (\ref{solution2}) are known to be non-smooth across
the negative $x$-axis if the domain was chosen to be the entire
circle. However, the present domain excludes the negative $x$-axis.
Thus, the source term $f$ of the Helmholtz equation (\ref{pde}) can
be simply defined as zero throughout $\Omega$. Nevertheless, there
still exists some singularities at the origin $(0,0)$. In
particular, it is remarked in \cite{Monk11} that the singularity
lies in the derivatives of the exact solution at $(0,0)$. Due to
such singularities, the convergence rates of high order
discontinuous Galerkin methods are also reduced for $\xi=3/2$ and
$\xi=2/3$ \cite{Monk11}. In the present study, we further note that
there exists a subtle difference between two cases $\xi=3/2$ and
$\xi=2/3$ at the origin. To see this, we neglect the second
$\cos(\cdot)$ term in the exact solution (\ref{solution2}) and plot
the Bessel function of the first kind $J_{\xi}( k |r|)$ along the
radial direction $r$, see Fig. \ref{fig.origin}. It is observed that
the Bessel function of the first kind is non-smooth for the case
$\xi=2/3$, while it looks smooth across the origin for the case
$\xi=3/2$. Thus, it seems that the first derivative of $J_{3/2}( k
|r|)$ is still continuous along the radial direction. This perhaps
explains why the present WG method does not experience any order
reduction for the case $\xi=3/2$. In  \cite{Monk11}, locally refined
meshes were employed to resolve the singularity at the origin so
that the convergence rate for the case $\xi=2/3$ can be improved. We
note that local refinements can also be adopted in the WG method for
a better convergence rate. A study of WG with grid local refinement
is left to interested parties for future research.

\begin{figure}[!tb]
\centering \resizebox{2.55in}{2.1in}{\includegraphics{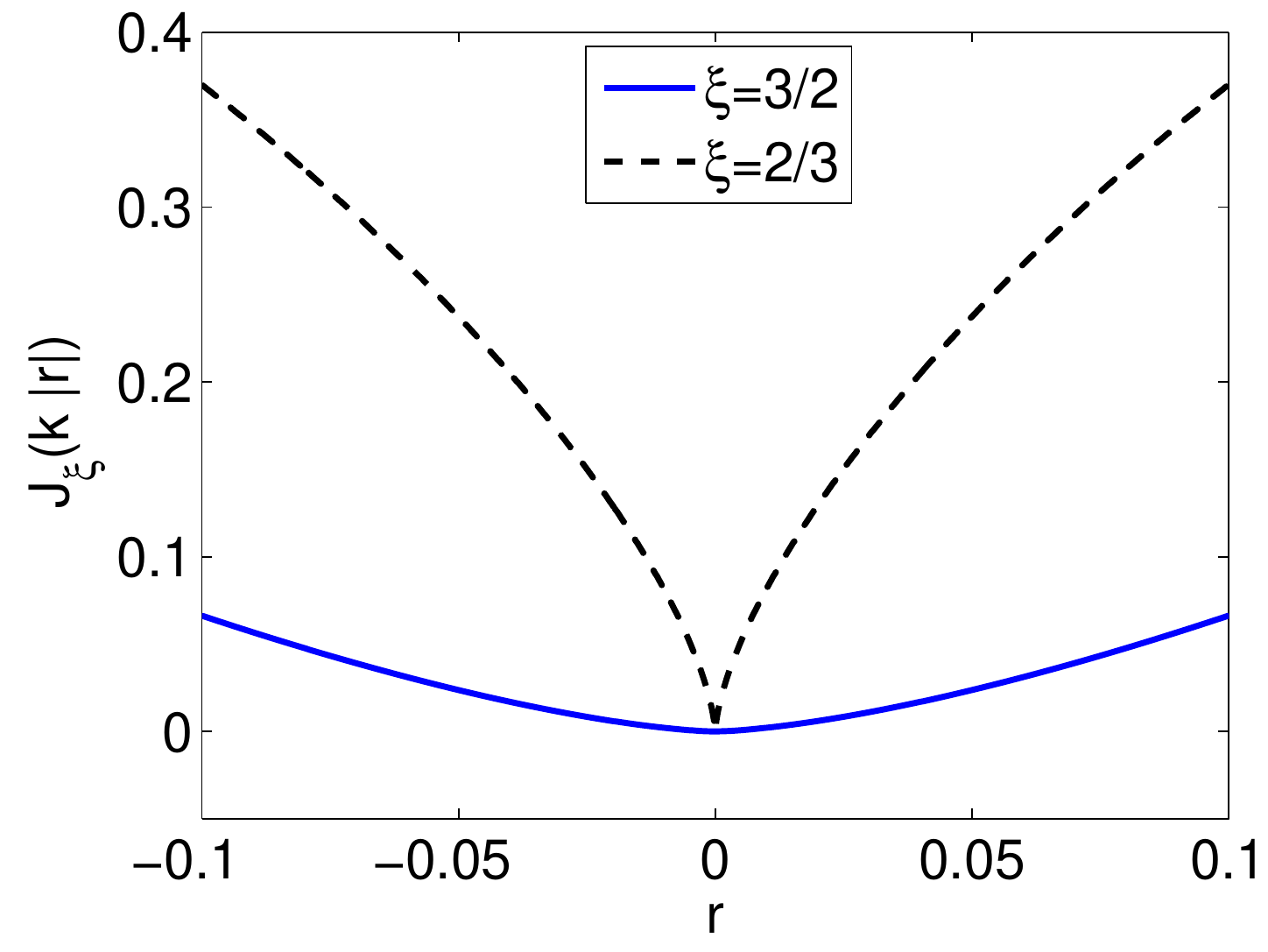}}
\caption{The Bessel function of the first kind $J_{\xi}( k |r|)$
across the origin.} \label{fig.origin}
\end{figure}

\subsection{A Helmholtz problem with inhomogeneous media}

We consider a Helmholtz problem with inhomogeneous media defined on
a circular domain with radius $R$. Note that the spatial function
$d(x,y)$ in the Helmholtz equation (\ref{pde}) represents the
dielectric properties of the underlying media. In particular, we
have $d=\frac{1}{\epsilon}$ in the electromagnetic applications
\cite{Zhao10}, where $\epsilon$ is the electric permittivity. In the
present study, we construct a smooth varying dielectric profile:
\begin{equation}\label{dr}
d(r)= \frac{1}{\epsilon_1}S(r) + \frac{1}{\epsilon_2}(1-S(r)),
\end{equation}
where $r=\sqrt{x^2+y^2}$, $\epsilon_1$ and $\epsilon_2$ are
dielectric constants, and
\begin{equation}
S(r)=
\begin{cases}
1 & \text{if $r<a$}, \\
-2\left( \frac{b-r}{b-a} \right)^3 +
 3\left( \frac{b-r}{b-a} \right)^2 & \text{if $a \le r \le b$}, \\
0 & \text{if $r>b$},
\end{cases}
\end{equation}
with $a<b<R$. An example plot of $d(r)$ and $S(r)$ is shown in Fig.
\ref{fig.eps}. In classical electromagnetic simulations, $\epsilon$
is usually taken as a piecewise constant, so that  some
sophisticated numerical treatments have to be conducted near the
material interfaces to secure the overall accuracy \cite{Zhao10}.
Such a procedure can be bypassed if one considers a smeared
dielectric profile, such as (\ref{dr}). We note that under the limit
$b \to a$, a piecewise constant profile is recovered in (\ref{dr}).
In general, the smeared profile (\ref{dr}) might be generated via
numerical filtering, such as the so-called $\epsilon$-smoothing
technique \cite{Shao03} in computational electromagnetics. On the
other hand, we note that the dielectric profile might be defined to
be smooth in certain applications. For example, in studying the
solute-solvent interactions of electrostatic analysis, some
mathematical models \cite{Chen10,Zhao11} have been proposed to treat
the boundary between the protein and its surrounding aqueous
environment to be a smoothly varying one. In fact, the definition of
(\ref{dr}) is inspired by a similar model in that field
\cite{Chen10}.

\begin{figure}[!tb]\label{fig.eps}
\centering \resizebox{2.55in}{2.1in}{\includegraphics{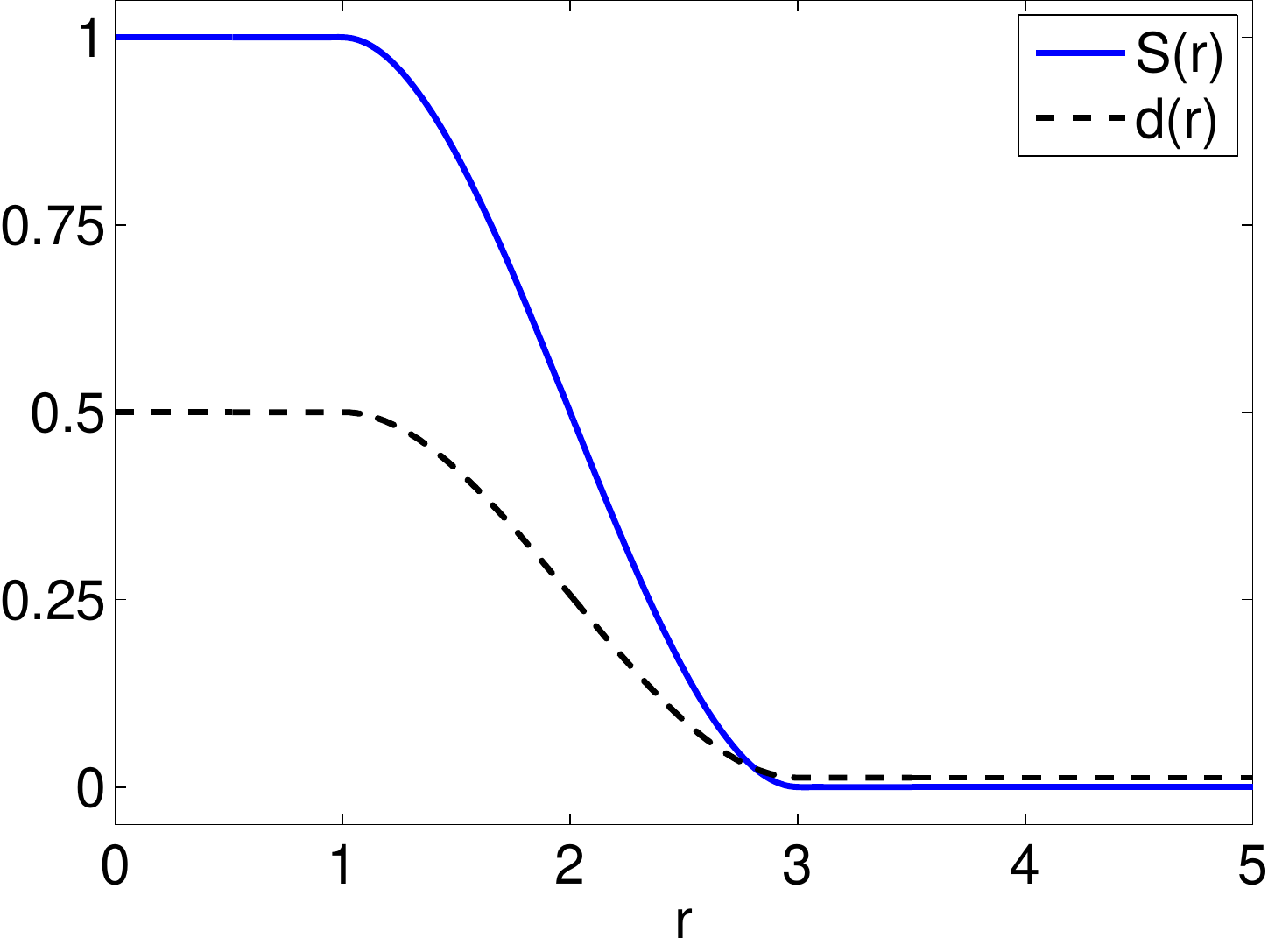}}
\caption{An example plot of smooth dielectric profile $d(r)$ and
$S(r)$ with $a=1$, $b=3$ and $R=5$. The dielectric coefficients of
protein and water are used, i.e., $\epsilon_1=2$ and
$\epsilon_2=80$. }
\end{figure}

In the present study, we choose the source of the Helmholtz equation
(\ref{pde}) to be
\begin{equation}
f(r)=\kappa^2 [d(r)-1] J_0 (k r) + k d'(r) J_1(kr),
\end{equation}
where
\begin{equation}
d'(r)=\left( \frac{1}{\epsilon_1}-\frac{1}{\epsilon_2} \right) S'(r)
\end{equation}
and
\begin{equation}
S'(r)=
\begin{cases}
0 & \text{if $r<a$}, \\
6\left( \frac{b-r}{b-a} \right)^2 -
6\left( \frac{b-r}{b-a} \right) & \text{if $a \le r \le b$}, \\
0 & \text{if $b<r$},
\end{cases}
\end{equation}
For simplicity, a Dirichlet boundary condition is imposed at $r=R$
with $u=g$. Here $g$ is prescribed according to the exact solution
\begin{equation}
u=J_0(kr).
\end{equation}

Our numerical investigation assumes the value of $a=1$, $b=3$ and
$R=5$. The wave number is set to be $k=2$. The dielectric
coefficients are chosen as $\epsilon_1=2$ and $\epsilon_2=80$, which
represents the dielectric constant of protein and water
\cite{Chen10,Zhao11}, respectively. The WG method with piecewise
constant finite element functions is employed to solve the present
problem with inhomogeneous media in Cartesian coordinate. Table
\ref{table.Ex3} illustrates the computational errors and some
numerical rate of convergence. It can be seen that the numerical
convergence in the relative $L^2$ error is not uniform, while the
relative $H^1$ error still converges uniformly in first order. This
phenomena might be related to the non-uniformity and smallness of
the media in part of the computational domain.
In particular, we note that
the relative $L^2$ error for the coarsest grid is extremely large, such that
the numerical order for the first mesh refinement is unusually high.
To be fair, we thus exclude this data in our analysis.
To have an idea about the overall numerical order of this non-uniform
convergence, we calculated the average convergence rate and least-square fitted
convergence rate for the rest mesh refinements,
which are $1.97$ and $1.88$, respectively.
Thus,
the present inhomogeneous example demonstrates the accuracy and robustness
of the WG method for the Helmholtz equation.

\begin{table}[!ht]
\caption{Numerical convergence test of the Helmholtz equation with
inhomogeneous media. } \label{table.Ex3}
\begin{center}
\begin{tabular}{|l|l|l|l|l|}
\hline
& \multicolumn{2}{c|}{relative $H^1$} & \multicolumn{2}{c|}{relative $L^2$} \\
\cline{2-3} \cline{4-5}
$h$ & error & order & error & order \\
\hline
   1.51e-00   & 2.20e-01 &       & 1.04e-00 &  \\
   7.54e-01   & 1.24e-01 & 0.83  & 1.20e-01 & 3.11 \\
   3.77e-01   & 6.24e-02 & 0.99  & 1.81e-02 & 2.73 \\
   1.88e-01   & 3.13e-02 & 1.00  & 5.71e-03 & 1.67 \\
   9.42e-02   & 1.56e-02 & 1.00  & 2.14e-03 & 1.42 \\
   4.71e-02   & 7.82e-03 & 1.00  & 5.11e-04 & 2.06 \\
\hline
\end{tabular}
\end{center}
\end{table}

\subsection{Large wave numbers}
We finally investigate the performance of the WG method for the
Helmholtz equation with large wave numbers.
As discussed above, without resorting to high order generalizations or
analytical/special treatments, we will examine the use of
the plain WG method for tackling the pollution effect.
The homogeneous
Helmholtz problem of the Subsection \ref{convex}
will be studied again. Also, the $RT_0$ and $RT_1$ elements are
used to solve the homogeneous Helmholtz equation with the Robin
boundary condition. Since this problem is defined on a structured
hexagon domain, a uniform triangular mesh with a constant mesh size
$h$ throughout the domain is used. This enables us to precisely
evaluate the impact of the mesh refinements. Following the
literature works \cite{bao04,fw}, we will focus only on the relative
$H^1$ semi-norm in the present study.

\begin{figure}[!ht]
\centering
\begin{tabular}{cc}
  \resizebox{2.45in}{2.15in}{\includegraphics{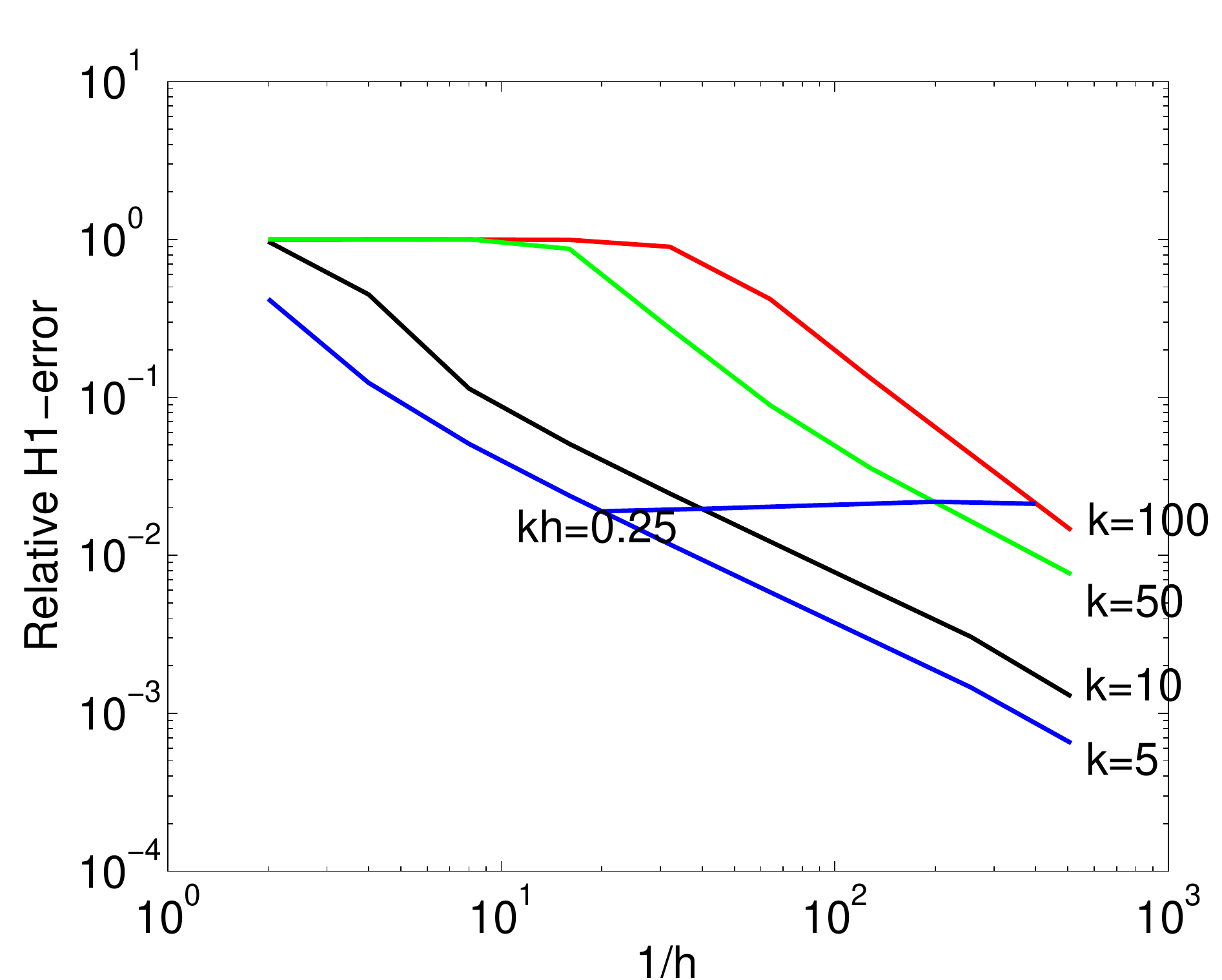}}
  \resizebox{2.45in}{2.15in}{\includegraphics{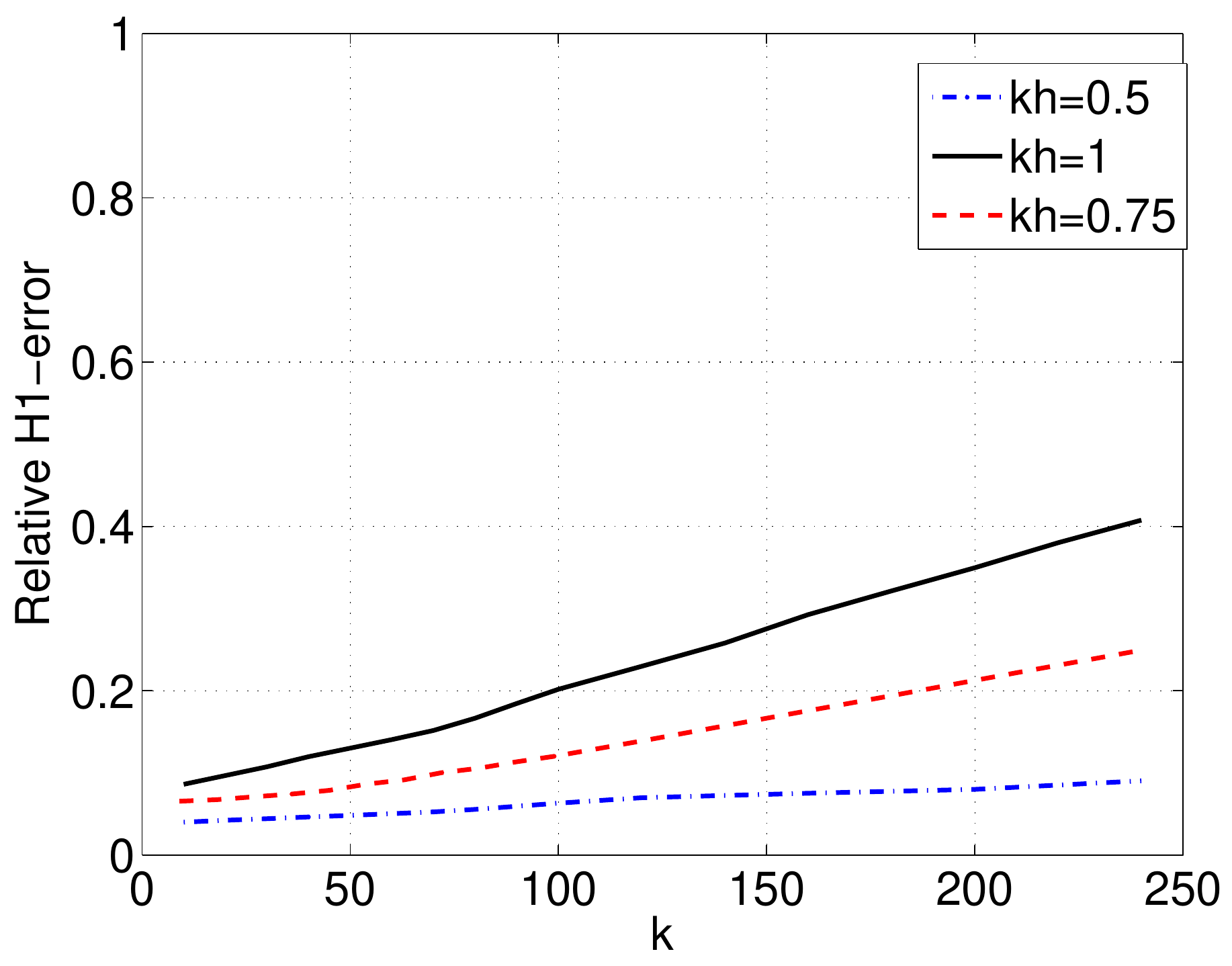}}
\end{tabular}
\caption{Relative $H^1$ error of the WG solution. Left: with respect
to $1/h$; Right: with respect to wave number $k$. } \label{fig.kh}
\end{figure}

To study the non-robustness behavior with respect to the wave number
$k$, i.e., the pollution effect, we solve the corresponding
Helmholtz equation by using piecewise constant WG method with
various mesh sizes for four wave numbers $k=5$, $k=10$, $k=50$, and
$k=100$, see Fig. \ref{fig.kh} (left) for the WG performance. From
Fig. \ref{fig.kh} (left), it can be seen that when $h$ is smaller,
the WG method immediately begins to converge for the cases $k=5$ and
$k=10$. However, for large wave numbers $k=50$ and $k=100$, the
relative error remains to be about $100$\%, until $h$ becomes to be quite
small or $1/h$ is large. This indicates the presence of the
pollution effect which is inevitable in any finite element method
\cite{BabSau}.
In the same figure, we also show the errors of
different $k$ values by fixing $kh=0.25$. Surprisingly, we found
that the relative $H^1$ error does not evidently increase as $k$
becomes larger. The convergence line for $kh=0.25$ looks almost
flat, with a very little slope. In other words, the pollution error
is very small in the present  WG result.
We note that such a result
is as good as the one reported in \cite{fw} by using a
penalized discontinuous Galerkin approach with optimized parameter
values.
In contrast, no parameters are involved in the WG scheme.

On the other hand, the good performance of the WG method for the
case $kh=0.25$ does not mean that the WG method could be free of
pollution effect. In fact, it is known theoretically  \cite{BabSau}
that the pollution error cannot be eliminated completely in two- and
higher-dimensional spaces for Galerkin finite element methods. In
the right chart of Fig. \ref{fig.kh}, we examine the numerical
errors by increasing $k$, under the constraint that $kh$ is a
constant. Huge wave numbers, up to $k=240$, are tested. It can be
seen that when the constant changes from $0.5$ to $0.75$ and $1.0$,
the non-robustness behavior against $k$ becomes more and more
evident. However, the slopes of $kh$=constant lines remain to be
small and the increment pattern with respect to $k$ is always
monotonic. This suggests that the pollution error is well controlled
in the WG solution.

\begin{figure}[!ht]
\centering
\begin{tabular}{cc}
  \resizebox{2.45in}{2.15in}{\includegraphics{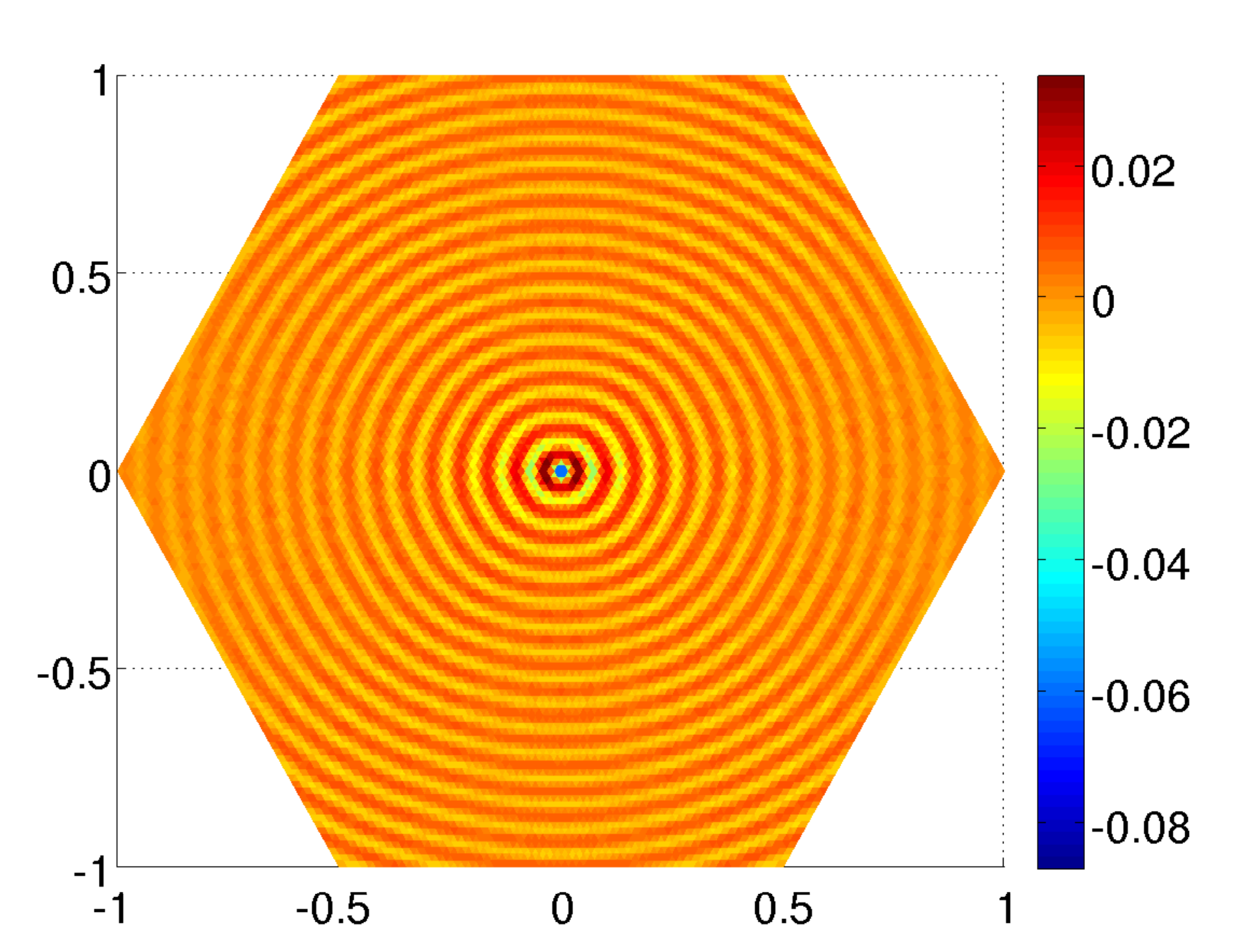}}
  \resizebox{2.45in}{2.15in}{\includegraphics{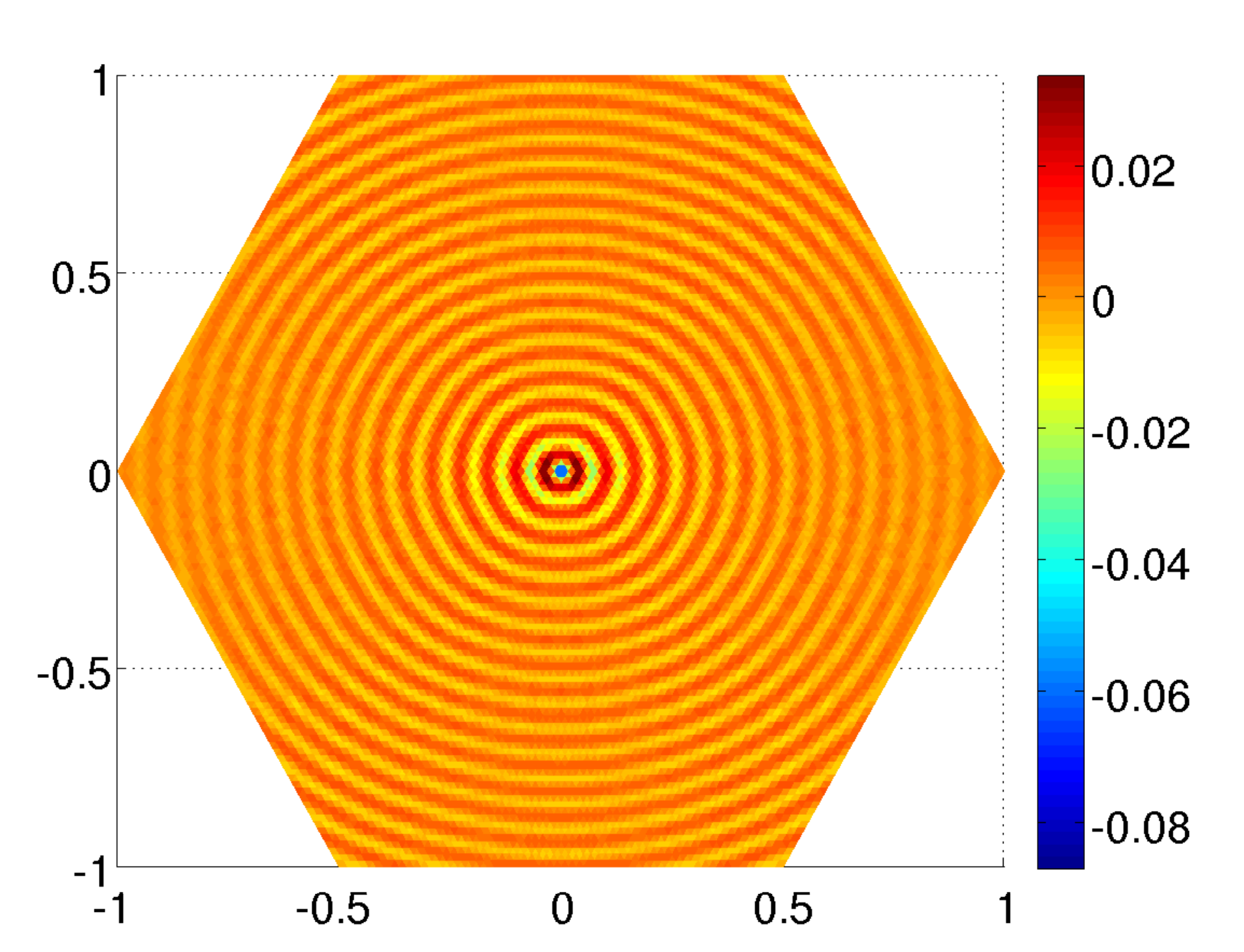}} 
\end{tabular}
\caption{Exact solution (left) and piecewise constant WG
approximation (right) for $k=100,$ and $h=1/60.$} \label{fig.solu2d}
\end{figure}

\begin{figure}[!ht]
\centering
\begin{tabular}{c}
  \resizebox{2.35in}{2.15in}{\includegraphics{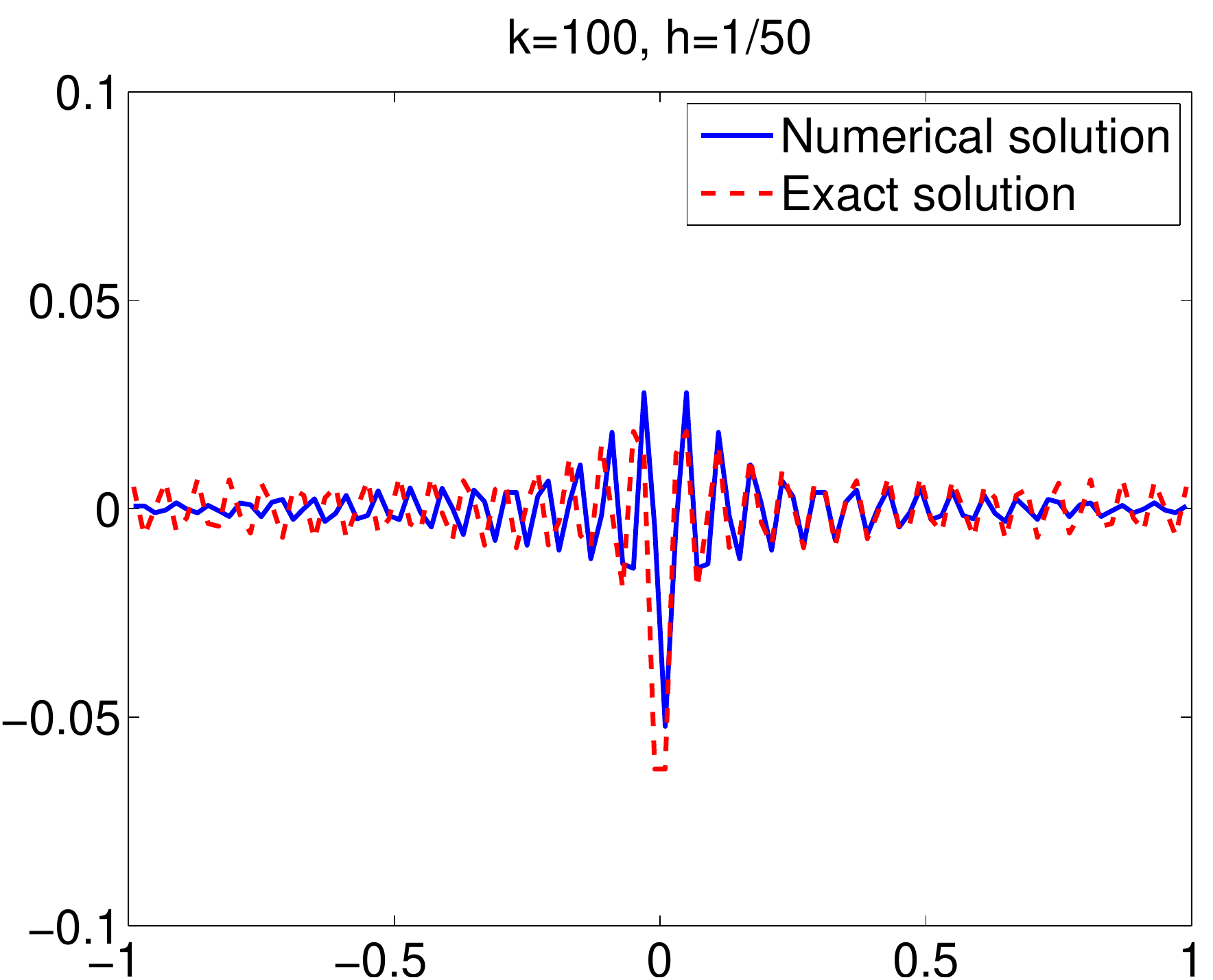}}\\
  \resizebox{2.35in}{2.15in}{\includegraphics{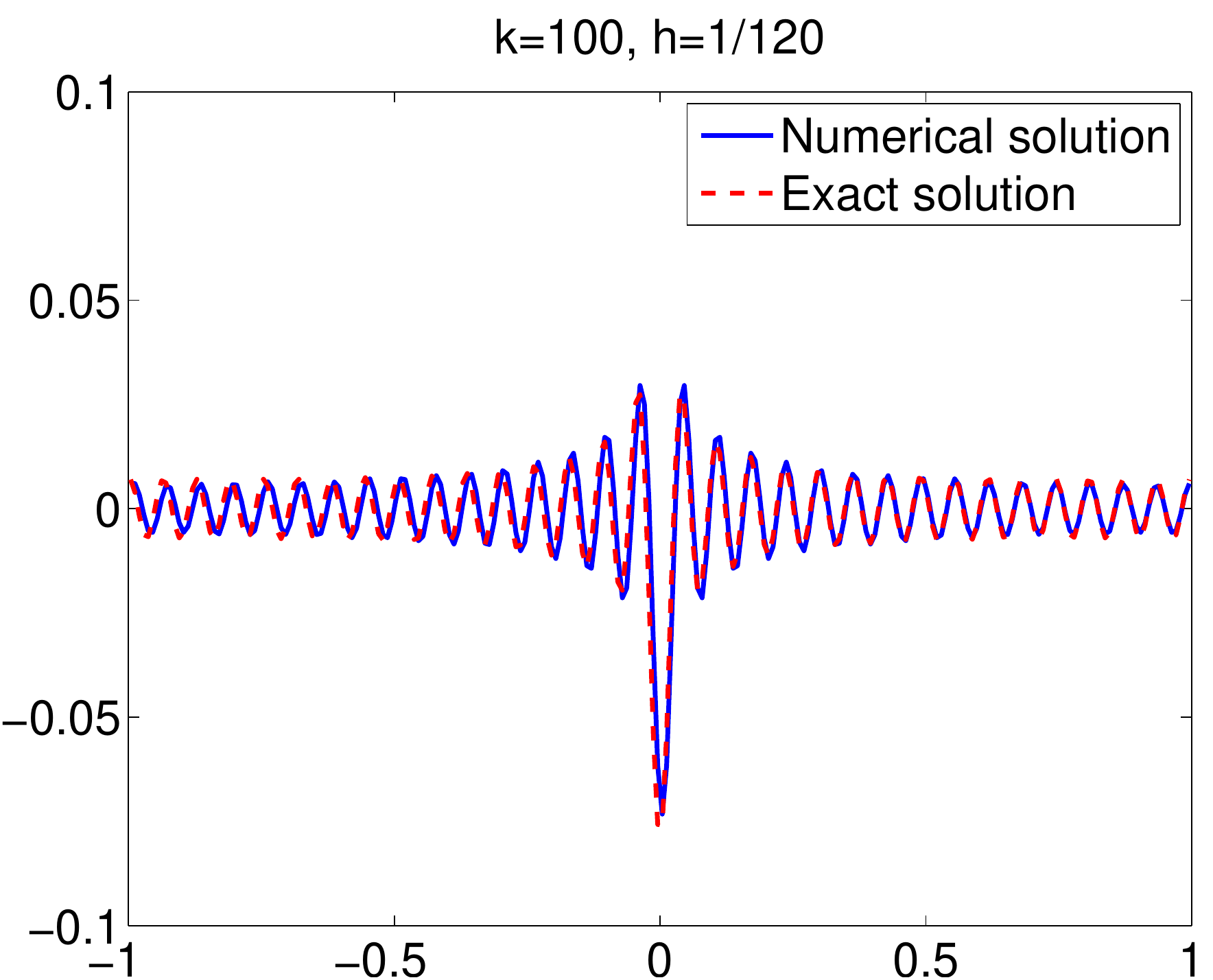}}\\
  \resizebox{2.35in}{2.15in}{\includegraphics{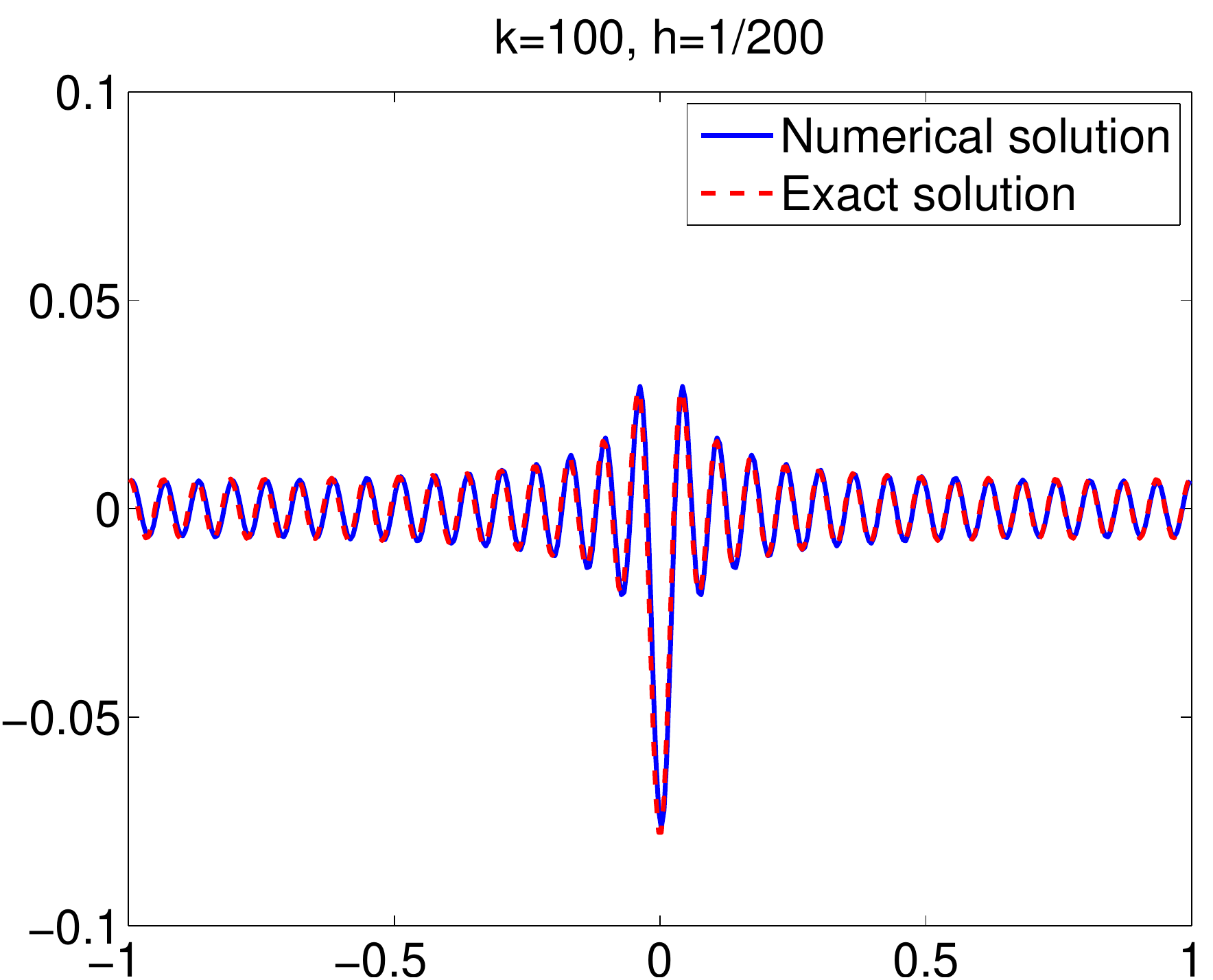}}
\end{tabular}
\caption{The trace plot along $x$-axis or $y=0$ form WG solution
using piecewise constants. } \label{fig.solu1d}
\end{figure}

\begin{figure}[!ht]
\centering
\begin{tabular}{cc}
  \resizebox{2.45in}{2.15in}{\includegraphics{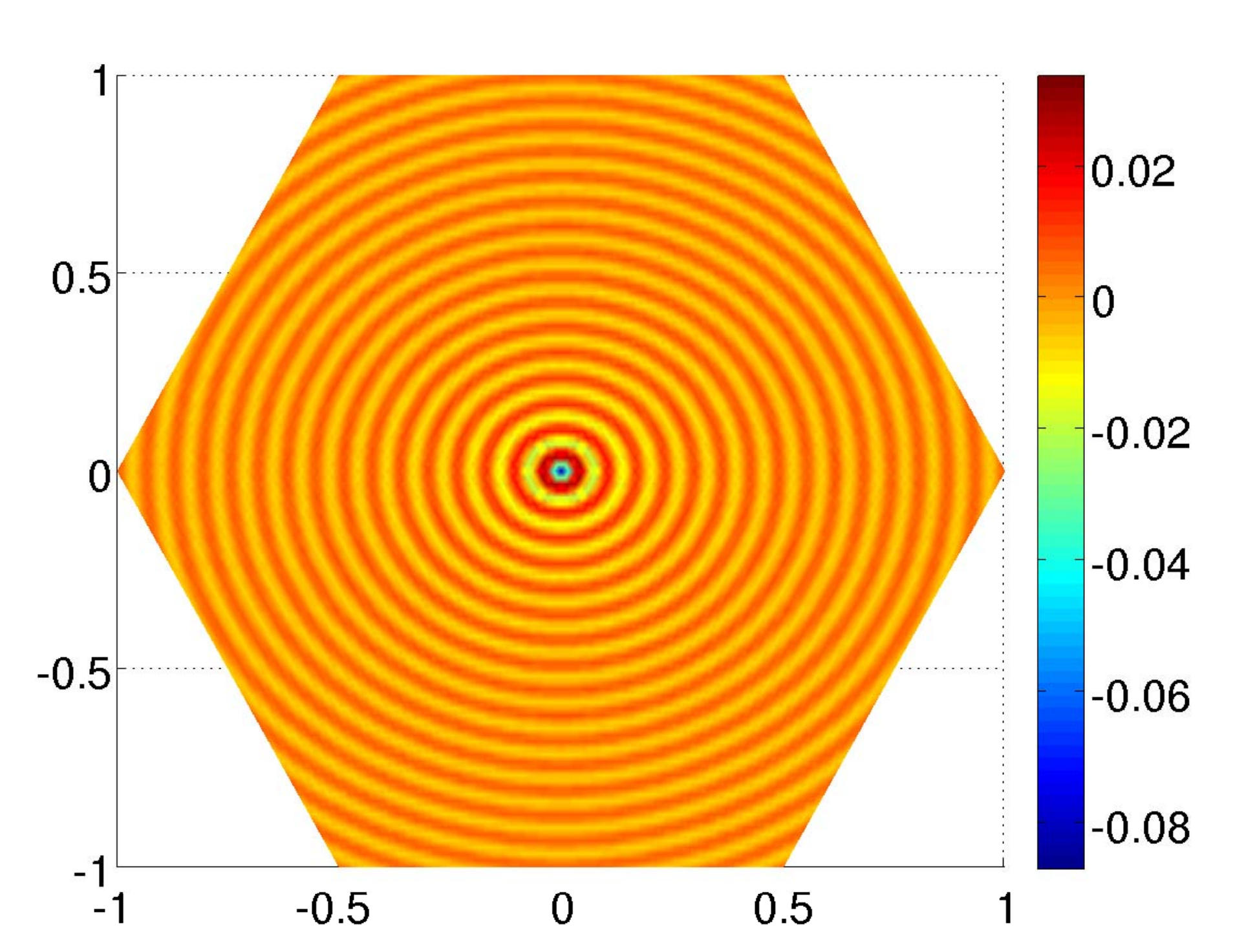}}
  \resizebox{2.45in}{2.15in}{\includegraphics{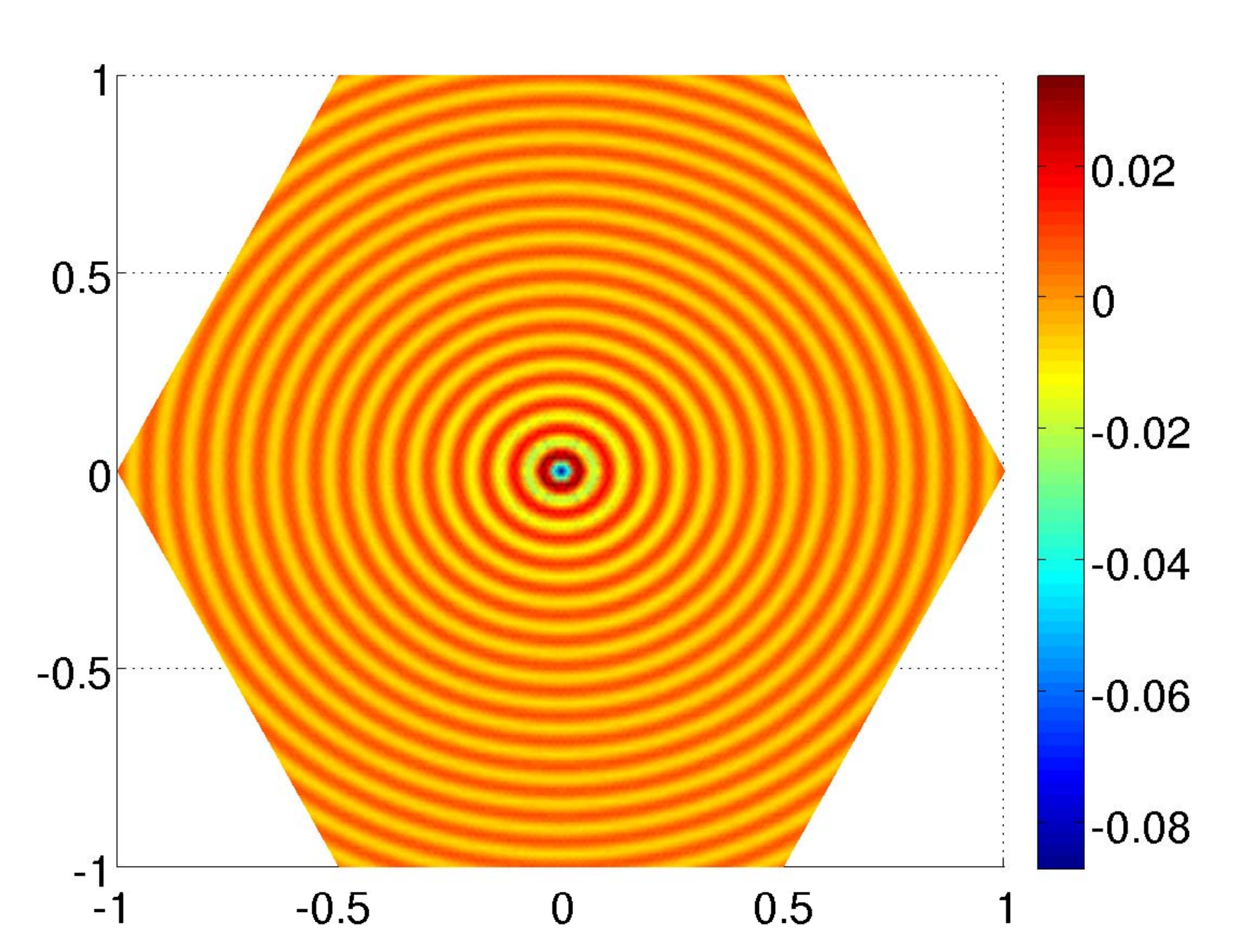}}
\end{tabular}
\caption{Exact solution (left) and piecewise linear WG approximation
(right) for $k=100,$ and $h=1/60.$} \label{fig.solu2d_linear}
\end{figure}

\begin{figure}[!ht]
\centering
\begin{tabular}{c}
  \resizebox{2.35in}{2.15in}{\includegraphics{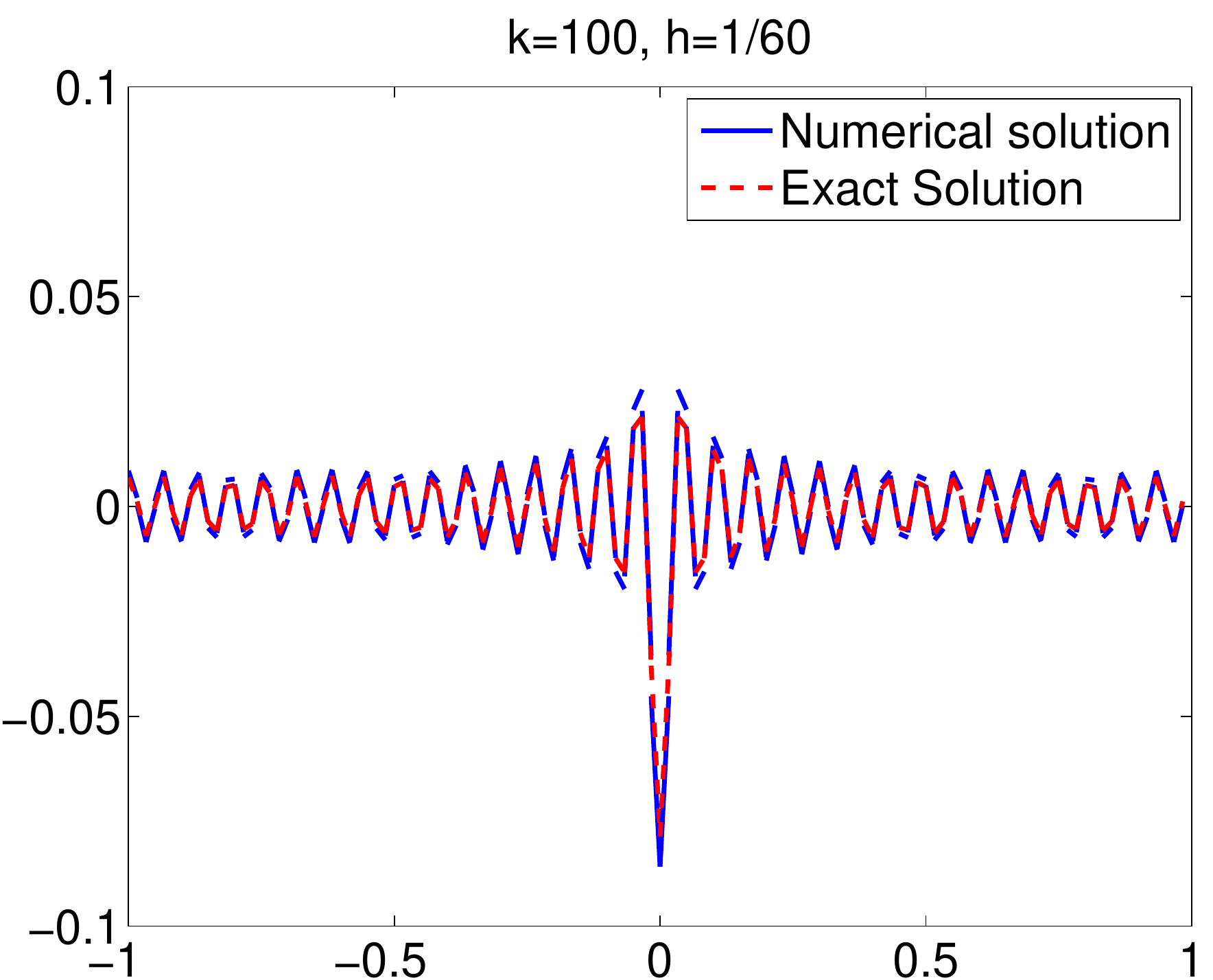}}\\
\end{tabular}
\caption{The trace plot along $x$-axis or $y=0$ form WG solution
using piecewise linear elements. } \label{fig.solu1d_linear}
\end{figure}

In the rest of the paper, we shall present some numerical results
for the WG method when applied to a challenging case of high wave
numbers. In Fig. \ref{fig.solu2d} and \ref{fig.solu2d_linear}, the
WG numerical solutions are plotted against the exact solution of the
Helmholtz problem. Here we take a wave number $k=100$ and mesh size
$h=1/60$ which is relatively a coarse mesh. With such a coarse mesh,
the WG method can still capture the fast oscillation of the
solution. However, the numerically predicted magnitude of the
oscillation is slightly damped for waves away from the center when
piecewise constant elements are employed in the WG method. Such
damping can be seen in a trace plot along $x$-axis or $y=0$. To see
this, we consider an even worse case with $k=100$ and $h=1/50$. The
result is shown in the first chart of Fig. \ref{fig.solu1d}. We note
that the numerical solution is excellent around the center of the
region, but it gets worse as one moves closer to the boundary. If we
choose a smaller mesh size $h=1/120$, the visual difference between
the exact and WG solutions becomes very small, as illustrate in Fig.
\ref{fig.solu1d}. If we further choose a mesh size $h=1/200$, the
exact solution and the WG approximation look very close to each
other. This indicates an excellent convergence of the WG method when
the mesh is refined. In addition to mesh refinement, one may also
obtain a fast convergence by using high order elements in the WG
method. Figure \ref{fig.solu1d_linear} illustrates a trace plot for
the case of $k=100$ and $h=1/60$ when piecewise linear elements are
employed in the WG method. It can be seen that the computational
result with this relatively coarse mesh captures both the fast
oscillation and the magnitude of the exact solution very well.

\section{Concluding Remarks}

The present numerical experiments indicate that the WG method as introduced
in \cite{wy} is a very promising numerical technique for solving the
Helmholtz equations with large wave numbers. This finite element
method is robust, efficient, and easy to implement. On the other
hand, a theoretical investigation for the WG method should be
conducted by taking into account some useful features of the
Helmholtz equation when special test functions are used. It would
also be valuable to test the performance of the WG method when high
order finite elements are employed to the Helmholtz equations
with large wave numbers in two and three dimensional spaces.

Finally, it is appropriate to clarify some differences and connections
between the WG method and other discontinuous finite element methods for solving
the Helmholtz equation.
Discontinuous functions are used to approximate the Helmholtz equation in many other finite element methods such as  discontinuous Galerkin (DG) methods \cite{fw,abcm,Chung}
and  hybrid discontinuous Galerkin (HDG) methods \cite{cdg,cgl,Monk11}.

However, the WG method and the HDG method are fundamentally different in concept and
formulation. The HDG method is formulated by using the standard
mixed method approach for the usual system of first order equations,
while the key to the WG is the use of the discrete weak differential
operators. For a second order elliptic problem, these two methods share the same feature by
approximating first order derivatives or fluxes through a formula
that was commonly employed in the mixed finite element method. For
high order partial differential equations (PDEs),
the WG method is greatly different from the HDG.
Consider the biharmonic equation \cite{mwy-biharmonic} as an example.
The first step of the HDG formulation is to rewrite the fourth order equation to
four first order equations. In contrast, the WG formulation for the biharmonic equation
can be derived directly from the variational form of the biharmonic equation
by replacing the Laplacian operator $\Delta$ by a weak Laplacian $\Delta_w$ and
adding a parameter free stabilizer \cite{mwy-biharmonic}.
It should be emphasized that the concept of
weak derivatives makes the WG a widely applicable numerical technique
for a large variety of PDEs  which we shall report in forthcoming papers.

For the Helmholtz equation studied in this paper,
the WG method and the HDG method
yield the same variational form for the homogeneous Helmholtz equation
with a constant $d$ in (\ref{pde}).
However, the WG discretization differs from the HDG discretization
for an inhomogeneous media problem with $d$ being a spatial function of
$x$ and $y$.
Moreover, the WG method has an advantage over the HDG method
when the coefficient $d$ is degenerated.

\end{document}